\documentclass[twoside,reqno]{amsart}
\usepackage[colorlinks=true,citecolor=blue]{hyperref}
\usepackage{mathptmx, amsmath, amssymb, amsfonts, amsthm, mathptmx, enumerate, color,mathrsfs}
\usepackage{graphicx}
\usepackage{epstopdf}
\newtheorem{theorem}{Theorem}[section]
\newtheorem{lemma}[theorem]{Lemma}
\newtheorem{proposition}[theorem]{Proposition}

\theoremstyle{definition}
\newtheorem{definition}[theorem]{Definition}

\newtheorem{example}[theorem]{Example}

\numberwithin{equation}{section}
\begin{document}
	\setcounter{page}{1}		
	\vspace*{2.0cm}
	\title[The Neumann problem for a class of generalized Kirchhoff-type potential systems ]
	{The Neumann problem for a class of  generalized \\
	Kirchhoff-type potential systems}
\author[N. Chems Eddine and D.D. Repov\v{s}]{Nabil Chems Eddine$^{1}$ 
and 
Du\v{s}an D. Repov\v{s}$^{2,3,4,*}$}
	\maketitle
	\vspace*{-0.6cm}	
	\begin{center}
	{\footnotesize
		$^1$Laboratory of Mathematical Analysis and Applications, Department of Mathematics, Faculty of Sciences, Mohammed V University, P.O. Box 1014, Rabat, Morocco.\\
		$^2$Faculty of Education, University of Ljubljana, Ljubljana, Slovenia.\\
		$^3$Faculty of Mathematics and Physics, University of Ljubljana,  Ljubljana, Slovenia.\\
		$^4$ Institute of Mathematics, Physics and Mechanics, Ljubljana, Slovenia.		
}\end{center}		
	\vskip 4mm {\footnotesize \noindent {\bf Abstract.}		
		In this  paper, we are concerned with the Neumann problem for a class of quasilinear stationary Kirchhoff-type potential systems, which involves general variable exponents elliptic operators with critical growth and real positive parameter. We show that the problem has at least one solution, which converges to zero in the norm of the space as the real positive parameter tends to infinity, via combining the truncation technique, variational method, and the concentration-compactness principle for variable exponent under suitable assumptions on the nonlinearities.\\
		
		
		}	
	
\renewcommand{\thefootnote}{}
\footnotetext{ $^*$Corresponding author: Du\v{s}an D. Repov\v{s}
	\par
	E-mail addresses: nab.chemseddine@gmail.com (N. Chems Eddine), 
	dusan.repovs@guest.arnes.si (D.D. Repov\v{s})\par }
	
	\section{Introduction}\label{s1}
	In this article, we are concerned with the existence and asymptotic behavior of nontrivial solutions for the following class of nonlocal quasilinear elliptic systems 
\begin{equation}\label{s1.1}
	\begin{gathered}
			M_i\Big(\mathcal{A}_i(u_i)\Big)\Big(
			-\textrm{div}\,\Big( \mathcal{B}_{1_i}(\nabla u_i)\Big) +  \mathcal{B}_{2_i}(u_i)\Big)=|u_i|^{s_i (x)-2}u_i+\lambda F_{u_i}(x,u) \quad \text{ in }\Omega, \\
			M_i\Big(\mathcal{A}_i(u_i)\Big)\mathcal{B}_{1_i}(\nabla u_i).\mathfrak{N}_i=|u_i|^{\ell_i (x)-2}u_i \qquad  \text{ on }\partial\Omega,
		\end{gathered}
	\end{equation}
	for $i=1,2,\dots,n$  ($n\in \mathbb{N}^\ast $), where $\Omega \subset \mathbb{R}^N (N\geq 2)$ is a bounded domain with smooth boundary $\partial \Omega$,  $\mathfrak{N}_i$ is
	the outward normal vector field on $\partial \Omega,$ $\lambda$ is a positive parameter,  $\nabla F=(F_{u_1},\dots,F_{u_2})$ is the gradient of a $C^1$-function $F: \mathbb{R}^N\times\mathbb{R}^n\to \mathbb{R}$,   $p_i, q_i, w_i,s_i \in C(\overline{\Omega}),$ and $\ell_i\in C(\partial \Omega)$ are such that
	\begin{gather}
		1 <p_i^{-}\leq p_i(x) \leq p_i^{+}< q_i^{-}\leq q_i(x) \leq q_i^{+}
		< N,
	\end{gather}
	\begin{gather}
		h_i^{-}\leq h_i(x) \leq h_i^{+}\leq w_i^{-}\leq w_i(x) \leq w_i^{+}\leq s_i^{-}\leq s_i(x) \leq s_i^{+}\leq h_i^{\ast}(x)<\infty,
	\end{gather}
	and
	\begin{gather}
		h_i^{-}\leq h_i(x) \leq h_i^{+}\leq w_i^{-}\leq w_i(x) \leq w_i^{+}\leq \ell_i^{-}\leq \ell_i(x) \leq \ell_i^{+}\leq h_i^{ \partial }(x)<\infty,
	\end{gather}
	for all $x\in \overline{\Omega}$, where functions $w_i$ are given by
	 condition $(\textbf{\textit{F}}_4)$ below, $p_i^-:= \inf_{x\in \overline{\Omega}}p_i(x)$,  $p_i^{+}:= \sup_{x\in \overline{\Omega}}p_i(x)$, and analogously to $w_i^-, w_i^+, q_i^-, q_i^+, h_i^-$,
	$h_i^+, s_i^-, s_i^+, \ell_i^-$ and $\ell_i^+$, with $h_i(x) =(1-\mathcal{K}(k^3_{i})) p_i(x)+ \mathcal{K}(k^3_{i})q_i(x),$ where $k^3_{i}$ is given by condition $(\textbf{\textit{A}}_2)$ below, and
	\[
	h_i^{\ast }(x)=\begin{cases}
		\frac{N h_i(x)}{N-h_i(x)} &\text{for } 	h_i(x)<N ,\\
		+\infty &\text{for }	h_i(x)\geq N,
	\end{cases} \quad
	\text{    and   } \quad
	h_i^{\partial  }(x)=\begin{cases}
		\frac{(N-1) h_i(x)}{N-h_i(x)} &\text{for } 	h_i(x)<N ,\\
		+\infty &\text{for }	h_i(x)\geq N,
	\end{cases}
	\]
	for all $x\in \overline{\Omega}$,  and the function $\mathcal{K}:\mathbb{R}_{0}^{+}\to \left\lbrace 0,1\right\rbrace  $ is defined by	
	\[
	\mathcal{K}(k_i)=\begin{cases}
		1 & \text{ if }     \   	k_i>0,\\
		0 & \text{ if }      \     k_i<0.
	\end{cases}
	\]	
	In addition, we consider both $\mathcal{C}^1_{h_{i}}$ and $\mathcal{C}^2_{h_{i}} $ as nonempty
	disjoint sets, which are respectively defined by	
	\[\mathcal{C}^1_{h_{i}}:= \lbrace x \in \partial\Omega,~  \ell_i (x) =h_i^{\partial}(x) \rbrace	 \
\text{ and } \
	\mathcal{C}^2_{h_{i}}:= \lbrace x \in \overline{\Omega}, ~~s_i (x) =h_i^\ast(x) \rbrace.
	\]
	The operators $\mathcal{A}_{j_i} : X_i\to \mathbb{R}^n$ for $j=1$ or $2$, and the operators  $ \mathcal{B}_i : X_i\to \mathbb{R}$, are respectively defined by  
\begin{equation}\label{AB}
	\mathcal{B}_{j_i}(u_i)= a_{j_i}(| u_i|^{p_i(x)}) | u_i|^{p_i(x)-2} u_i, 
	\
	\text{ and }
	\
	\mathcal{A}_i(u_i)= \displaystyle\int_{\Omega }\dfrac{1}{p_i(x)}\Big(A_{1_i}(|\nabla u_i|^{p_i(x)})+ A_{2_i}(|u_i|^{p_i(x)})\Big)dx,
\end{equation}
for all $1\leq i\leq n$, where  $X_i:= W^{1,h_i(x)}(\Omega)\cap W^{1,p_i(x)}(\Omega)$ is a Banach space, and functions $A_{j_i}$ are defined by $A_{j_i}(t)=\displaystyle\int_{0}^{t}a_{j_i}(k)d k$, where functions $a_{j_i}$ are described in condition $(\textbf{\textit{A}}_1)$.\\	
  In what follows, we shall consider the functions $ a_{j_i}$  satisfying the following assumptions for all $i\in \{1,2,\dots,n\},j\in \{1,2\}:$	
	\begin{itemize}
		\item[$(\textbf{\textit{A}}_1)$]  $a_{j_i}:\mathbb{R}^+ \to \mathbb{R}^+$ are of class $C^1$.		
		\item[$(\textbf{\textit{A}}_2)$] There exist positive constants $k_{j_i}^0, k_{j_i}^1, k_{j_i}^2$ and $k_i^3$  for all $i\in \{1,2,\dots,n\}, j\in \{1,2\}$ such that		
		$$k_{j_i}^0 + \mathcal{K}(k_i^3)k_{j_i}^2 \xi^{\frac{q_i(x)-p_i(x)}{p_i(x)}} \leq a_{j_i}(\xi) \leq k_{j_i}^1 + k_i^3\xi^{\frac{q_i(x)-p_i(x)}{p_i(x)}},
			\		
		\hbox{for all}
		\
		\xi \geq 0
			\
		\hbox{ and  a.e.} 
		\
		x\in \overline{\Omega}.$$		
		\item[$(\textbf{\textit{A}}_3)$] There exists $c_i>0$  for all $1\leq i\leq n$ such that		
		$$ \min \left\lbrace a_{j_i}(\xi^{p_i(x)}) \xi^{p_i(x)-2}, a_{j_i}(\xi^{p_i(x)})\xi^{p_i(x)-2}
		+\xi \frac{\partial(a_{j_i}(\xi^{p_i(x)})\xi^{p_i(x)-2})}{\partial \xi } \right\rbrace \geq c_i\xi^{p_i(x)-2},
		\
	\hbox{for a.e.} 
	\
	x\in \overline{\Omega}
	\		
	\hbox{and  all}
	\
	\xi \geq 0.$$			
		\item[$(\textbf{\textit{A}}_4)$]
		There exist positive constants $\beta_{j_i}$ and $\gamma_i$ for all $i\in \left\lbrace 1,2,\dots,n\right\rbrace,
		j\in \{1,2\}$ such that		
		$$ A_{j_i}(\xi)\geq \frac{1}{\beta_{j_i}}a_{j_i}(\xi)\xi \text{ with } h_i^+<\gamma_i <s_i^-
		\text{ and } \frac{q_i^+}{p_i^+}\leq \beta_{j_i} < \frac{\gamma_i}{p_i^+},\		
		\hbox{for all}
		\
		\xi \geq 0.$$
		\item[$(\textbf{\textit{M}})$] $M_i: \mathbb{R}^+ \longrightarrow \mathbb{R}$ are continuous and increasing functions such that $M_i(t)\geq M_i(0)=\mathfrak{M}_i^0>0 $, for all $t \geq 0$, $i\in \left\lbrace 1,2,\cdots,n\right\rbrace $. 
	\end{itemize}
As it is well known, there are many examples of functions $ M_i $  that satisfy  assumption  $ (\textbf{\textit{M}})$, for example
$$
M_1(\xi)=\mathfrak{M}_1^0 + \mathfrak{B}_1\xi^{\theta_1}, 
\
\text{ with } 
\
 \mathfrak{M}_1^0, \mathfrak{B}_1\geq 0, \mathfrak{M}_1^0+\mathfrak{B}_1>0 
 \
 \text{    and   }
 \
\theta_1 \geq 1.
$$
In particular, when $\mathfrak{M}_1^0=0$ and $\mathfrak{B}_1>0$, the Kirchhoff equation associated with $M_1$ is said to be degenerate. On the other hand, when
 $\mathfrak{M}_1^0>0$ and $\mathfrak{B}_1\geq0$, the Kirchhoff equation associated with $M_1$ is said to be  nondegenerate. In this case, when $\mathfrak{B}_1=0$, the Kirchhoff equation associated with $M_1$ (is a constant) reduces to a local quasilinear elliptic problem.\\  
 The study of differential equations and variational problems driven by nonhomogeneous differential operators has received extensive attention and has been extensively investigated,
 see  e.g., Papageorgiou et al. \cite{Papageorgiou}. This is due to their ability to model many physical phenomena. It should be noted that the $p(x)$-Laplacian operator is a special case of the divergence form operator $\textrm{div}\,\Big( \mathcal{B}_{j_i}(\nabla u_i)\Big)$. The natural functional framework for this operator is described by the Sobolev space with a variable exponent $W^{1,p(x)}$. \\ 
 In recent decades, there has been a particular focus on variable exponent Lebesgue and Sobolev spaces, $L^{p(x)}$ and $W^{1,p(x)}$, where $p$ is a real function, e.g.,  R\v{a}dulescu and Repov\v{s} \cite{radulescu-Repovs}. Traditional Lebesgue spaces $L^p$ and Sobolev spaces $W^{1,p}$ with constant exponents have proven insufficient to tackle the complexities of nonlinear problems in the applied sciences and engineering. To address these limitations, the use of variable exponent Lebesgue and Sobolev spaces has been on the rise in recent years. \\ 
 This area of research reflects a new type of physical phenomena, such as electrorheological fluids, or "smart fluids," which can exhibit dramatic changes in mechanical properties in response to an electromagnetic field. These and other nonhomogeneous materials require the use of variable exponent Lebesgue and Sobolev spaces, where the exponent $p$ is allowed to vary. Moreover, variable exponent Lebesgue and Sobolev spaces have found a wide range of applications, from image restoration and processing to fields such as thermorheological fluids, mathematical biology, flow in porous media, polycrystal plasticity, heterogeneous sand pile growth, and fluid dynamics. For a comprehensive overview of these and other applications, see e.g. Chen et al.  \cite{Chen1}, Diening et al.  \cite{Dien0,Dien1},  Halsey  \cite{Halsey}, R\v{a}dulescu \cite{radulescu}, Ru\v{z}i\u{c}ka  \cite{Ru1,Ru2}, and the references therein).\\	
	Furthermore, every single equation of  system
	\eqref{s1.1} is a generalization of the stationary problem of the first model introduced by Kirchhoff  \cite{Kir} in 1883 of the following form:	
	\begin{equation}
		\label{e1.1}
		\rho\partial_{tt}^2u-\left(\dfrac{\rho_0}{h}+\dfrac{E}{2L}\int_{0}^{L}\left\vert \partial_{x}u(x)\right\vert^2dx \right)\partial_{xx}^2u=0,
	\end{equation}	
 where the parameters $\rho$, $h$ , $\rho_0$, $t$, $L$, $E$ are constants with some physical meaning,  which is an extension of the classical D'Alembert wave equation, by considering the effect of the change in the length of a vibrating string. \\
	Nearly a century later, in 1978, Jacques-Louis Lions  \cite{Lions} returned to the equation and proposed a general Kirchhoff equation in arbitrary dimension with external force term which was written as	
			\begin{equation*}\label{s1.3}
				\begin{gathered}
			\partial_{tt}^2u-\left(
			a+b \displaystyle\int_{\Omega}\left| \nabla u\right|^2dx \right)\Delta u=f(x,u) \quad \text{ in }\Omega, \\
			u=0 \quad\text{on }\partial\Omega.
		\end{gathered}
	\end{equation*}	
Later on, many interesting results have been obtained, see e.g., Caristi et al.  \cite{Caristi}, Dai and Hao  \cite{Dai1}, Ma  \cite{Ma} and the references therein. The main difficulty in studying these equations appears that they do not satisfy a pointwise identity any longer. It is generated by having the term containing $M_i$ in the equations, and it makes \eqref{s1.1} a nonlocal problem.  The nonlocal problem models arise in the description of biological systems and also various physical phenomena, where $u$ describes a process that depends on the average of itself, such as population density. For more references on this subject, we refer the interested reader to Ambrosio et al. \cite{Ambrosio}, Arosio and Pannizi  \cite{Arosio},  Cavalcanti et al. \cite{Cavalcanti}, Chipot and Lovat  \cite{Chipo}, He et al. \cite{He-Qin-Wu},  Corr\^{e}a and Nascimento  \cite{Corr2}, Yang and Zhou \cite{Yang}, Wang et al. \cite{Wang} and the references therein.	
On the one hand, the differential equations with constant or variable critical exponents in bounded or unbounded domains have attracted increasing attention recently. They were first discussed in the seminal paper by Brezis and Nirenberg  \cite{Brezis} in 1983, which treated Laplacian equations. Since then, there have been extensions of  \cite{Brezis} in many directions. \\
One of the main features of elliptic equations involving critical growth is the lack of compactness arising in connection with the variational approach. In order to overcome the lack of compactness, Lions  \cite{Lions1} established the method using the so-called concentration compactness principle (CCP, for short) to show a minimizing sequence or a Palais-Smale ((PS), for short) sequence is precompact. Afterward, The variable exponent version of the Lions concentration-compactness principle for a bounded domain was independently obtained by Bonder et al.  \cite{Bonder,Bonder2}, Fu  \cite{Fu}, and  for an unbounded domain by Fu  \cite{Fu1}. Since then, many authors have applied these results to study critical elliptic problems involving variable exponents, see e.g., Alves et al.  \cite{Alves0,Alves}, Chems Eddine et al.  \cite{chemsRagusa,chems1,chems}, Fang and Zhang  \cite{Fang1}, Hurtado et al.  \cite{Hurtado},  Mingqi et al. \cite{Mingqi}, Zhang and  Fu  \cite{Zhang}.\\	
	When $M_i$ satisfies conditions
	$a_{1_i}\equiv1$ (with $k^1_{i}=1$,   and $k^2_{i}>0$ and $k^3_{i}=0$), and $a_{2_i}\equiv0$, Chems Eddine  \cite{chems1} proved the existence of nontrivial  weak solutions for  the following class of Kirchhoff-type potential systems with Dirichlet boundary conditions	
	\begin{equation*}
		\begin{gathered}
			-M_i\left(\displaystyle\int_{\Omega  }\dfrac{1}{p_i(x)}|\nabla u_i|^{p_i(x)}\right)
			\Delta_{p_i(x)}u_i=|u_i|^{q_i (x)-2}u_i +\lambda F_{u_i}(x,u)  \quad \text{ in }\Omega, \\
			u_i=0 \quad\text{on }\partial\Omega.
	\end{gathered}
\end{equation*}
$1\leq i\leq n (n\in \mathbb{N}^\ast)$, where $p_i,q_i : \overline{\Omega}\to \mathbb{R} $ are Lipschitz continuous functions such that $1\leq  q_i (x)\leq p_i^{\ast}(x)$ for all $x$ in $\Omega$ and the potential function $F$ satisfies mixed and subcritical growth conditions.		
	Following that, Chems Eddine  \cite{chems0}
	established the existence of an infinite of solutions for the following systems	
\begin{equation*}
	\begin{gathered}
			-M_i\Big(\mathcal{A}_i(u_i)\Big)
			\textrm{div}\,\Big( \mathcal{B}_i(\nabla u_i)\Big)=|u_i|^{s_i (x)-2}u_i +\lambda F_{u_i}(x,u)  \text{ in }\Omega, \\
			u=0 \quad \text{on }\partial\Omega,
			\end{gathered}
	\end{equation*}
	for $i=1,2,\dots,n$  ($n\in \mathbb{N} $), where $\lbrace x \in \Omega, ~~s_i (x) =p_i^\ast(x) \rbrace $ nonempty and  $F\in C^1(\Omega\times\mathbb{R}^n,\mathbb{R})$  satisfies some mixed and subcritical growth conditions. Next, Chems Eddine and Ragusa  \cite{chemsRagusa} dealt with cases when the above class of Kirchhoff-type potential systems under Neumann boundary conditions with two critical exponents, established the existence and multiplicity of solutions.	\\	
  Our objective in this article is to show the existence of nontrivial solutions for the nonlocal problem \eqref{s1.1}. As we shall see in this paper, there are some  main difficulties in our situation, which
   can be summed up in three main problems. First, the assumption $(\textbf{\textit{M}})$ provides only a positive lower bound
    for the Kirchhoff functions $M_i$ near zero, creating
    serious mathematical technical difficulties.
    To overcome these difficulties, we need to do a truncation of
     the Kirchhoff functions $M_i$ to obtain prior estimates of the boundedness from above, and thus obtain a new auxiliary problem, thereafter, we do another truncation to control the energy functional corresponding to
      the auxiliary problem. The second difficulty in solving problem \eqref{s1.1} is the lack of compactness which can be illustrated by the fact that the embeddings
       $W^{1,h(x)}(\Omega) \hookrightarrow L^{h^{*}(x)}(\Omega)$ and  $W^{1,h(x)}(\Omega) \hookrightarrow L^{h^{\partial}(x)}(\partial \Omega)$ are no longer compact
    and, to overcome this difficulty, we use two versions of Lions's principle for the variable exponent extended by Bonder et al. \cite{Bonder,Bonder2}. Then  by combining the variational method and Mountain  Pass Theorem, we obtain the existence of at least one nontrivial solution to the auxiliary problem, see Theorem \ref{AuxTh}, and by truncating functions $M_i$, we obtain the existence of at least one nontrivial solution to  problem \eqref{s1.1}, see Theorem \ref{thm.result}.\\	
	Throughout this paper, we shall assume that $F$ satisfies the following conditions:	
	\begin{itemize}
		\item[($\textbf{\textit{F}}_1$)] $F\in C^{1}(\Omega \times \mathbb{R}^n,\mathbb{R})$ and $F\big(x,0_{\mathbb{R}^n}\big)=0$	.	
		\item[($\textbf{\textit{F}}_2$)]   
		For all $(i,j)\in \left\lbrace 1,2,\dots,n\right\rbrace^2$,
		there exist positive functions $b_{i_j}$ such that
		\begin{gather*}
			\big|  F_{\xi_i}(x,\xi_1,\dots,\xi_n)\big| \leq
			\sum_{j=1}^{n} b_{i_j}(x)| \xi_j|^{r_{i_j}(x)-1},
		\end{gather*}
		where
		$1<r_{i_j}(x)<\inf_{x\in \Omega}h_i(x)$ for all $x\in \Omega $. The weight-functions $b_{ii}$ (resp  $b_{i_j}$ if $i\neq j$) belong to the generalized Lebesgue spaces $L^{\alpha_{ii}}(\Omega)$ (resp $L^{\alpha_{i_j} }(\Omega)$), with
		\[
		\alpha_{ii}(x)=\frac{h_i(x)}{h_i(x)-1}, \quad \alpha_{i_j} (x)
		=\frac{h_i^{\ast }(x)h_j^{\ast }(x)}{h_i^{\ast }(x)h_j^{\ast }(x)-h_i^{\ast
			}(x)-h_j^{\ast }(x)}.
		\]						
		\item[($\textbf{\textit{F}}_3$)] There exist $K>0$ and $ \gamma_i \in (h_i^+, \inf\{s_i^-, \ell_i^-\})$ for all $(x,\xi_1,\dots,\xi_n)\in \Omega\times\mathbb{R}^n$ where $| u_i|^{\gamma _i}\geq K $		
		\[
		0 < F(x,\xi_1,\dots,\xi_n)< \sum_{i=1}^{n}  \frac{\xi_i}{\gamma_i} F_{\xi_i}(x,\xi_1,\dots,\xi_n).
		\] 					
		\item[($\textbf{\textit{F}}_4$)]	There exists a positive constant $c$ such that
		\[
		|F(x,\xi_1,\dots,\xi_n)| \leq c\Big (\sum_{i=1}^{n} | \xi_i|^{w_i(x)}\Big),  \ \hbox{for all} \ (x,\xi_1,\dots,\xi_n)\in  \Omega\times\mathbb{R}^n,
		\]
		where 	$w_i \in C_{+}(\overline{\Omega})$ and  $q_i^+< w_i^-\leq w_i^+<  \inf\{s_i^-, \ell_i^-\} \leq \inf\{s_i^+, \ell_i^+\},  \ \hbox{for all} \ 1\leq  i\leq n $.		
	\end{itemize}
	\begin{example}
		 There are many potential functions $F$ satisfying assumptions $(\textbf{\textit{F}}_1)$ and $(\textbf{\textit{F}}_2)$.  For example, when $n=2,$ take  
		$
		F(x,u_1,u_2)=b_{12}(x)|u_1|^{r_1(x)}|u_2|^{r_2(x)},
		$
		where $\frac{r_1(x)}{h_1(x)}+\frac{r_2(x)}{h_2(x)}< 1$, and the positive weight-function $b_{12} \in L^{\alpha(x)}(\Omega) $ with		
		 \[
		 \alpha(x)= \dfrac{h_2^{\ast}(x)h_1^{\ast}(x)}{h_2^{\ast}(x)h_1^{\ast}(x) -h_2^{\ast}(x)r_1(x) -h_2^{\ast}(x) r_2(x),}
		 	\		
		 \hbox{ for all }
		 \
		x\in \Omega .	 
		 \] 
		 By the standard calculus we can verify that $F$ satisfies the assumption $(\textbf{\textit{F}}_1)$. Moreover, by using  Young inequality we can check the assumption $(\textbf{\textit{F}}_2)$.\end{example}
	The main result  of our paper is the following.	
	\begin{theorem} \label{thm.result}
	Assume that conditions $(\textbf{\textit{A}}_1)-(\textbf{\textit{A}}_4)$, $(\textbf{\textit{M}}),$ and $(\textbf{\textit{F}}_1)-(\textbf{\textit{F}}_4)$ hold. Then there exists $\lambda^{\star} >0$ such that for all $\lambda \geq \lambda_{\star}$, problem \eqref{s1.1} has at least one nontrivial solution in $X$. Moreover, if $u_\lambda$ is a weak solution of problem \eqref{s1.1}, then 
		$\lim_{ \lambda \to +\infty } \| u_\lambda\|=0.$
	\end{theorem}
	The  paper is organized as follows: In Section \ref{sec2} we give some preliminary results of the variable exponent spaces. In Section \ref{sec auxiliary} we introduce the auxiliary problem and obtain a nontrivial solution for the auxiliary problem. Section \ref{proof res} is dedicated to proving the main results. Finally, in Section \ref{Examples}, we illustrate the degree of generality of the kind of problems we studied in this paper. 	
	\section{Preliminaries and basic notations}\label{sec2}	
	In this section, we introduce some definitions and results which will be used in the next section.	
Throughout  our work, let $\Omega$ be a bounded domain of $\mathbb{R}^N$ $(N\geq 2)$ with a Lipschitz boundary $\partial \Omega$,
and let us denote by $\mathcal{D}$ either $\Omega$
 or its boundary $\partial \Omega$.
Denote
\begin{equation*}
	\mathcal{M}_{+}(\mathcal{D}):= \Big\{  p: \mathcal{D} \to \mathbb{R}  \ \text{measurable real-valued function}: p(x)> 1 \text{ for a.a. } x \in \mathcal{D} \Big\},  	
\end{equation*}
\begin{equation*}
	C_{+}(\mathcal{D}):= \Big\{  p \in C(\mathcal{D}): p(x)> 1 \quad \text{ for a.a. } x \in \mathcal{D} \Big\}.
\end{equation*}
For all $ p \in \mathcal{M}_{+}(\mathcal{D}),$  denote
$p^{+}:= \sup_{x\in \mathcal{D}}p(x)$
	 and
	 $ p^{-}:= \inf_{x\in \mathcal{D}}p(x)$
and
for all $p \in \mathcal{M}_{+}(\mathcal{D})$ and for a measure $\eta$ on $\mathcal{D}$, we define the variable exponent Lebesgue space as
\begin{gather*}
	L^{p(x)}(\mathcal{D}):=L^{p(x)}(\mathcal{D},d\eta):= \Big\{ u \ \text{measurable real-valued function}:  \rho_{p,\mathcal{D}}(u)  <\infty \Big\},
\end{gather*}
where the functional $\rho_{p,\mathcal{D}}:L^{p(x)}(\mathcal{D}) \to \mathbb{R}$ is defined as
$$ \rho_{p,\mathcal{D}}(u):= \int_{\Omega}|u(x)|^{p(x)}d\eta.$$
The functional $\rho_{p,\mathcal{D}}$ is called the $p(x)$-modular of the $L_{ }^{p(x)}(\mathcal{D})$ space, it
has played  an important role in manipulating the generalized Lebesgue-Sobolev spaces. We endow the space $L_{}^{p(x)}(\mathcal{D})$ with the Luxemburg norm
\begin{equation*}
	\|u\|_{L^{p(x)}(\mathcal{D})}: = \inf \left\{ \tau >0 ; \rho_{p,\mathcal{D}}\bigg(\frac{u(x)}{\tau}\bigg) \leq 1 \right \}.
\end{equation*}
Then $(L^{p(x)}(\mathcal{D}),\|u\|_{L^{p(x)}(\mathcal{D})})$ is a separable and reflexive Banach space (see, e.g., Kov\'{a}\v{c}ik and R\'{a}kosn\'{i}k  \cite[Theorem 2.5, Corollary 2.7]{Kov1}). In the subsequent sections, the $L^{p(x)}$-spaces under consideration will be
$L^{p(x)}(\Omega):=L^{p(x)}(\Omega, dx)$, and 	$L^{p(x)}(\partial \Omega):=L^{p(x)}(\partial\Omega, d\sigma)$ for an appropriate measure $\sigma$
supported
on $\partial \Omega$. Let us now recall more basic properties concerning the Lebesgue spaces.
\begin{proposition}[Kov\'{a}\v{c}ik and R\'{a}kosn\'{i}k {\cite[Theorem 2.8]{Kov1}}] \label{prop21}
	Let  $p$ and $q$ be variable exponents in $\mathcal{M}_{+}(\mathcal{D})$ such that $p\leq q$ in $\mathcal{D}, $ where $0<\text{meas}(\mathcal{D})<\infty$. Then the embedding $L^{q(x)}(\mathcal{D})\hookrightarrow L^{p(x)}(\mathcal{D})$ is continuous.
\end{proposition}
Furthermore, the following H\"older-type inequality
\begin{equation}\label{Hol}
	\Big| \int_{\mathcal{D}}u(x)v(x)dx\Big|
	\leq \Big(\frac{1}{p^{-}}+ \frac{1}{(p')^{-}}\Big)\| u\|_{L^{p(x)}(\mathcal{D})}\| v\| _{L^{p'(x)}(\mathcal{D})}
	\leq 2\| u\|_{L^{p(x)}(\mathcal{D})}\| v\| _{L^{p'(x)}(\mathcal{D})}
\end{equation}
holds  for all  $u\in
L^{p(x)}({\mathcal{D}})$ and $v\in L^{p'(x)}({\mathcal{D}})$
(see, e.g., Kov\'{a}\v{c}ik and R\'{a}kosn\'{i}k \cite[Theorem 2.1]{Kov1}), where we denoted by
$L^{p'(x)}(\mathcal{D})$ the topological dual space (or the conjugate space) of
$L^{p(x)}(\mathcal{D})$, obtained by conjugating the exponent pointwise, that is,  $1/p(x)+1/p'(x)=1$ (see, e.g., Kov\'{a}\v{c}ik and R\'{a}kosn\'{i}k \cite[Corollary 2.7]{Kov1}). Moreover, if $h_1,h_2,h_3:\mathcal{D}\to (1,\infty)$ are Lipschitz continuous functions such that
\[
\frac{1}{h_1(x)}+\frac{1}{h_2(x)}+\frac{1}{h_3(x)}=1,
\]
then for all $u \in L^{h_1(x)}(\mathcal{D})$,$ v\in L^{h_2(x)}(\mathcal{D})$,
$ w\in L^{h_3(x)}(\mathcal{D}),$ the following inequality holds
\[
\int_{\mathcal{D}}| u(x)v(x)w(x)|dx
\leq \Big(\frac{1}{h_1^{-}}+ \frac{1}{h_2^{-}}+\frac{1}{h_3^{-}}\Big)\| u\|_{ L^{h_1(x)}(\mathcal{D})}\| v\| _{ L^{h_2(x)}(\mathcal{D})}\| w\| _{ L^{h_3(x)}(\mathcal{D})}.
\]
If $u\in L^{p(x)}({\mathcal{D}})$ and $p<\infty$, we have the following properties (see e.g., Fan and Zhao \cite[Theorem 1.3, Theorem 1.4]{Fan}):
\begin{equation}\label{L22}
	\|u\|_{L^{p(x)}({\mathcal{D}})}<1\;(=1;\,>1)\; \quad  \text{ if and only if } \quad \; \rho_{p,\mathcal{D}} (u) <1\;(=1;\,>1),
\end{equation}
\begin{equation}\label{L23}
	\text{ if }\quad 	\|u\|_{L^{p(x)}(\mathcal{D})}>1\; \quad \text{ then } \quad \;
	\|u\|_{L^{p(x)}(\mathcal{D})}^{p^-}\leq \rho_{p,\mathcal{D}}(u)
	\leq\|u\|_{L^{p(x)}(\mathcal{D})}^{p^+},
\end{equation}
\begin{equation}\label{L24}
	\text{ if }\quad 	\|u\|_{L^{p(x)}(\mathcal{D})}<1\; \quad \text{ then } \quad  \;
	\|u\|_{L^{p(x)}(\mathcal{D})}^{p^+}\leq \rho_{p,\mathcal{D}}(u)
	\leq\|u\|_{L^{p(x)}(\mathcal{D})}^{p^-}.
\end{equation}
As a consequence, we have the equivalence of modular  and norm convergence
\begin{equation}\label{L25}
	\|u\|_{L^{p(x)}(\mathcal{D})}\to 0\;(\to\infty)\; \quad \text{ if and only if } \quad \; \rho_{p,\mathcal{D}} (u) \to 0\; (\to\infty).
\end{equation}
	\begin{proposition}[Edmunds and Rakosnik \cite{Edmu1}] \label{prop22}
	Let  $h$ and $\ell$ be variable exponents in $\mathcal{M}_{+}(\mathcal{D})$ with $1\leq h(x), \ell(x)\leq \infty $ a.e. $x$ in
	$\Omega$ and  $p\in L^{\infty }(\Omega)$.  Then  if $u\in L^{\ell(x)}(\Omega)$, $u\neq 0$, it follows that
	\begin{gather*}
		\big\| u \big\| _{L^{h(x)\ell(x)}(\Omega)}\leq 1\Rightarrow 	\big\|u \big\| _{L^{h(x)\ell(x)}(\Omega)}^{h^{-}}
		\leq \big\|\big| u\big| ^{h(x)}\big\| _{L^{\ell(x)}(\Omega)}\leq \big\| u \big\| _{L^{h(x)\ell(x)}(\Omega)}^{h^{+}},
		\\
		\big\| u \big\| _{L^{h(x)\ell(x)}(\Omega)}\geq 1\Rightarrow \big\|u \big\| _{L^{h(x)\ell(x)}(\Omega)}^{p^{+}}
		\leq \big\| \big| u \big| ^{h(x)}\big\| _{L^{\ell(x)}(\Omega)}\leq \big\| u \big\| _{L^{h(x)\ell(x)}(\Omega)}^{h^{-}}.
	\end{gather*}
	 When  $h(x)=h$ is constant, we obtain 
$
	\big\| \big| u \big| ^{h}\| _{L^{\ell(x)}(\Omega)}
	=\big\|u \big\|_{L^{h\ell(x)}(\Omega)}^{h}.
	$
\end{proposition}
Now, let us pass to the Sobolev space with variable exponent, that is,
$$
W^{1,p(x)}(\Omega):=\Big\{ u\in L^{p(x)}(\Omega):  \partial_{x_i}u \in L^{p(x)}(\Omega) \text{ for } i =1,\dots,N\Big\},
$$
where $\partial_{x_i}u= \frac{\partial u}{\partial x_i}$ represent the partial derivatives of $u$ with respect to $x_i$ in the weak sense. This space has a corresponding modular given by

$$\rho_{1,p(x)} (u):=\int_{\Omega} | u|^{p(x)} + |\nabla u|^{p(x)}dx$$
which yields the norm
\begin{equation*}
	\|u\|_{1,p(x)}:=\| u\|_{W^{1,p(x)}(\Omega)}: = \inf \left\{ \tau >0 ;~~ \rho_{1,p}\bigg(\frac{u(x)}{\tau}\bigg) \leq 1 \right \}.
\end{equation*}
Another possible choice of norm in $W^{1,p(x)}(\Omega)$ is  $\|u\| _{L^{p(x)}(\Omega)}+\| \nabla u \| _{L^{p(x)}(\Omega)}$. Both norms turn out to be equivalent but we use the first one for convenience.
It is well known that  $ W^{1,p(x)}(\Omega)$ is a separable and reflexive Banach spaces (see, e.g., Kov\'{a}\v{c}ik and R\'{a}kosn\'{i}k \cite[Theorem 3.1]{Kov1}).
As usual, we define by  $p^{\ast }(x)$ the critical Sobolev exponent and  $p^{\partial }(x)$ the critical
Sobolev trace exponent, respectively by
\[
p^{\ast }(x):=\begin{cases}
	\frac{Np(x)}{N-p(x)} &\text{for } p(x)<N \\
	+\infty &\text{for }p(x)\geq N
\end{cases}
\quad \text{ and }
p^{\partial}(x):=\begin{cases}
	\frac{(N-1)p(x)}{N-p(x)} &\text{for } p(x)<N \\
	+\infty &\text{for }p(x)\geq N
\end{cases}
\]
We recall the following crucial embeddings on $ W^{1,p(x)}(\Omega)$
\begin{proposition}[Diening et al.  \cite{Dien1}, Edmunds and Rakosnik \cite{Edmu1}] \label{prop23}	
	Let $p$ be Lipschitz continuous and satisfying $1<p^-\leq p(x)\leq p^+<N$, and let $q\in C(\overline{\Omega})$ satisfy	
	$ 1\leq q(x) \leq p^{\ast}(x),$ for all
	$x\in \overline{\Omega}.$	
	Then there exists a continuous embedding
	$W^{1,p(x)}(\Omega)\hookrightarrow L^{q(x)}(\Omega).$
	If we assume in addition that $1\leq q(x)< p^{\ast}(x)$ for all $x\in \overline{\Omega}$, 
	then
	 this embedding is compact.	
\end{proposition}
\begin{proposition}[Diening et al.  \cite{Dien1}, Edmunds and Rakosnik \cite{Edmu1}] \label{prop24}
	Let $p\in W^{1,h}(\Omega)$ with $1\leq p_{-}\leq p_{+}<N<h$. Then for all $q\in C(\partial \Omega)$ satisfying
	$1\leq q(x) \leq  p^{\partial}(x)$  for $x\in \partial \Omega$, there is a continuous boundary trace embedding
	$W^{1,p(x)}(\Omega)\hookrightarrow L^{q(x)}(\partial \Omega).$
	If we assume in addition that $1\leq q(x)< p^{\ast}(x)$ for all $x\in \overline{\Omega}$,   then this embedding is compact.	
\end{proposition}
For detailed properties of the variable exponent Lebesgue-Sobolev spaces, we refer the reader to  Diening et al. \cite{Dien1}, Kov\'{a}\v{c}ik and R\'{a}kosn\'{i}k \cite{Kov1}.
As is well known, the use of critical point theory needs the well-known Palais–Smale condition ($(PS)_c$ for short) which plays a central role.
	\begin{definition}
	 Consider a function $E:X\to \mathbb{R}$ of class $C^1$, where $X$ is a real Banach space. We say that a sequence  $\{u_m\}$ is a Palais-Smale sequence
	 for the functional $E$ if
	\begin{equation}
		E(u_m) \to c   \text{ and } E'(u_m)\to 0 \text{ in } X',
	\end{equation}
	We say that $\{u_m\}$ is a Palais-Smale sequence with energy level $c$ (or $(u_m)$ is $(PS)_c$ for short). Moreover, if every $(PS)_c$ sequence for $E$ has a strongly convergent subsequence in $X$, then we say that $E$ satisfies the Palais-Smale condition at level $c$ (or $E$ is $(PS)_c$ short).
\end{definition}
  Our main tool is the following classical Mountain Pass Theorem.
\begin{theorem}[Rabinowitz  \cite{Rabinowitz}] \label{PMT}
	Let $X$ be a real infinite dimensional Banach space and  let $E: X \to \mathbb{R}$ be of class $C^{1}$ and satisfying the $(PS)_c$ such that $E(0_X)=0$. Assume that	
	\begin{itemize}
		\item[($H_1$)] there exist positive constants $ \rho,\mathcal{R} $ such that $E(u)\geq\mathcal{R},$  for all $u\in \partial B_\rho\cap X,$
		\item[($H_2$)]  there exists $z\in X$ with $\left\| z\right\|_{X} >\rho$ such that $E(z)<0$.
	\end{itemize}
	Then
	 $E$  has a critical value $c\geq \mathcal{R}$, which can be characterized as
	$c := \inf_{\phi \in \Gamma} \max_{ \delta \in \left[ 0,1\right]  }E(\phi(\delta)),$
	where
	$$\Gamma= \big\{ \phi : \left[ 0,1\right] \to X\ \text{continuous}: \phi(0)=0_X, E(\phi (1))<0 \big\}. $$
\end{theorem}
In the sequel, we shall use the product space
$
X:=\prod_{i=1}^{n}\Big(W^{1,h_i(x)}(\Omega)\cap W^{1,p_i(x)}(\Omega)\Big),
$
equipped with the norm
$
\| u\|:=\max_{1\leq i \leq n}\big\{  \|  u_i\|_{i} \big\} ,$
for all
$ u=(u_1,u_2,\dots,u_n)\in X,
$
where $\| u_i\|_{i}:= \| u_i\|_{1,p_i(x)}+ \mathcal{K}(k_i^3)\|  u_i\|_{1,q_i(x)}$, is the norm of $ W^{1,h_i(x)}(\Omega)\cap W^{1,p_i(x)}(\Omega), $ for all $i \in \{1,2,\dots,n\}$. 
\begin{definition}
	 We say that $u=(u_1,u_2,\dots,u_n) \in X$ is a weak solution of the system  \eqref{s1.1} if 
	\begin{align*}
		\sum_{i=1}^{n} M_i\left(
		\mathcal{A}_i(u_i)\right)\int_{\Omega  } \Big(
		\mathcal{B}_{1_i}(\nabla u_i)\nabla v_i + \mathcal{B}_{2_i}(u_i) v_i
		\Big)\,dx - \sum_{i=1}^{n}\int_{\Omega  }|u_i|^{s_i(x)-2}u_i v_i dx  
		- \sum_{i=1}^{n}\int_{\partial \Omega  }|u_i|^{\ell_i(x)-2}u_i v_i\,d\sigma_x\\	- \sum_{i=1}^{n}\displaystyle\int_{\Omega  }\lambda F_{u_i}(x,u)v_i\,dx =0,
	\end{align*}
	for all $v=(v_1,v_2,\dots,v_n)\in X= \prod_{i=1}^{n}(W^{1,h_i(x)}(\Omega)\cap W^{1,p_i(x)}(\Omega))$.
\end{definition}
The energy functional  $E_{\lambda}:X\longrightarrow\mathbb{R} $ associated with problem \eqref{s1.1} is defined as follows: 
$E_{\lambda}(.):= \Phi(.)- \Psi (.) -\Upsilon(.)-\mathcal{F}_{\lambda}(.),$
where $\Phi, \Psi \text{ and } \mathcal{F}_{\lambda}:X\longrightarrow\mathbb{R} $
are given by
\begin{align*}
	\Phi(u)=\sum_{i=1}^{n}\widehat{M_i}\left(
	\mathcal{A}_i(u_{i}(x))\right),	&\quad
	\Psi(u)=	\sum_{i=1}^{n}\displaystyle\int_{\Omega  }\dfrac{1}{s_i(x)}|u_i|^{s_i(x)}dx, \\
	\Upsilon(u)=	\sum_{i=1}^{n}\displaystyle\int_{\partial\Omega  }\dfrac{1}{\ell_i(x)}|u_i|^{\ell_i(x)}d\sigma_x& \text{ and }
	\mathcal{F}_{\lambda}(u)=\int_{\Omega}\lambda F(x,u)dx,
\end{align*}
for all $u=(u_1,\dots,u_n)$ in $X$, where 
$
\widehat{M}_{i}(\tau)= \int_{0}^{\tau}M_{i}(s)ds.
$
By  standard calculus, one can see that under the above assumptions, the energy functional $E_{\lambda}:X\to \mathbb{R}^N$ corresponding to problem (\ref{s1.1}) is well defined and  $E_{\lambda}\in C^1(X,\mathbb{R})$ with
\begin{align*}
	\left\langle E_\lambda'(u),v\right\rangle &= 	\sum_{i=1}^{n} M_i\left(
	\mathcal{A}_i(u_i)\right)\int_{\Omega  } \Big(
	\mathcal{B}_{1_i}(\nabla u_i)\nabla v_i + \mathcal{B}_{2_i}(u_i) v_i
	\Big)\,dx - \sum_{i=1}^{n}\int_{\Omega  }|u_i|^{s_i(x)-2}u_i v_i \,dx \\
	&\quad- \sum_{i=1}^{n}\int_{\partial \Omega  }|u_i|^{t_i(x)-2}u_i v_i\,d\sigma_x
	- \sum_{i=1}^{n}\displaystyle\int_{\Omega  }\lambda F_{u_i}(x,u)v_i\,dx ,
\end{align*}
for all $v=(v_1,v_2,\dots,v_n)\in X$.  So  the critical points of functional $E_{\lambda}$ are weak solutions of system \eqref{s1.1}. \\
To prove our existence result, since we have lost 
compactness in the inclusions  $W^{1,h_i(x)}(\Omega)\hookrightarrow L^{h_i^\ast(x)}(\Omega)$ and $W^{1,h_i(x)}(\Omega)\hookrightarrow L^{h_i^\partial(x)}(\partial\Omega),$ for all $i$ in $\{1,2,\dots,n\},$ we can no longer expect the Palais–Smale condition to hold. Nevertheless, we can prove a local Palais–Smale condition that will hold for the energy functional $E_{\lambda}$ below a certain value of energy, by using the principle of concentration compactness for the variable exponent Sobolev space $W^{1,h_i(x)}(\Omega)$. For  reader's convenience, we state this result in order to prove Theorem \ref{thm.result}, see Bonder et al.  \cite{Bonder, Bonder2} for its proof.\\
	Now, let $\mathcal{O}$ be a different subset of $\partial \Omega$, a  closed set (possibly empty). Set	
	$$
		W^{1,h(x)}_{\mathcal{O}}(\Omega):= \overline{\{ v \in C^{\infty}(\overline{\Omega }) : \text{ $v$ vanishes on a neighborhood of $\mathcal{O}$ }\}}$$
		where closure is taken with respect to  $\ \big\| v\big\| _{1,h(x)}.$
	This is the subspace of functions vanishing on $\mathcal{O}$. Evidently, $W^{1,h(x)}_{\emptyset}(\Omega) =W^{1,h(x)}(\Omega)$.  In general, $W^{1,h(x)}_{\mathcal{O}}(\Omega)=W^{1,h_i(x)}(\Omega)$
	if and
	only if the $h_i(x)$-capacity of $\mathcal{O}$ equals zero, for more details we refer the interested readers to Harjulehto et al. \cite{Har}. The best Sobolev trace constant
	$T_i(h_i(x),\ell_i(x),\mathcal{O})$ is defined by	
	$$ 0<T_i(h_i(x),\ell_i(x),\mathcal{O}):=\inf_{v\in W^{1,h_i(x)}_{\mathcal{O}}(\Omega)}
	\frac{\|  v  \|_{W^{1,h_i(x)}(\Omega)}}{\|v \|_{L^{\ell_i(x)}(\partial\Omega)}}.$$	
	\begin{theorem}[Bonder et al.  \cite{Bonder2}]\label{ccp}		
		Let $h_i\in C_{+}(\overline{\Omega})$, $\ell_i \in C_{+}(\partial \Omega)$ be such that $\ell_i(x)\leq h_i^{ \partial }(x),  \ \hbox{for all} \ x \in \partial \Omega$, 
		and $\{u_m\}_{m\in\mathbb{N}}$ be a sequence in
		$W^{1,h_i(x)}(\Omega)$  such that $u_{i_m} \rightharpoonup u_i$
		 weakly  in $W^{1,h_i(x)}(\Omega)$.		 
Then
 there exist a countable  index set $J_i^1$, positive numbers
$ \{\mu_{i_j}\}_{j \in J_i^1}$ and $ \{\nu_{i_j}\}_{j \in J_i^1},$ and 
$\{x_j\}_{j\in J_i^1}\subset \mathcal{C}_{h_i}^1= \{x\in \partial\Omega\colon \ell_i(x)=h_i^{\partial}(x)\}$ such that 
		\begin{gather}
		|u_{i_m}|^{\ell_i(x)}\rightharpoonup	\nu_{i}=|u_i|^{\ell_i(x)} + \sum_{j\in J_i^1}\nu_{i_j}\delta_{x_j}\ \text{  weakly-*  in the sense of measures,}\\\label{2.1a1}
		|\nabla u_{i_m}|^{h_i(x)}\rightharpoonup	\mu_i \geq |\nabla u_i|^{h_i(x)} + \sum_{j\in I} \mu_{i_j} \delta_{x_j} \ \ \text{  weakly-*
			in the sense of measures,}\\ \label{2.1a2}
			\overline{ T_i }_{x_j} \nu_{i_j}^{1/h_i^{\partial}(x_j)} \leq \mu_{i_j}^{1/h_i(x_j)} \ \ \hbox{for all} \ j\in J_i^1,
		\end{gather}
		where 
		\begin{equation}\label{GNS}
			\overline{ T_i }_{x_j} :=\sup_{\epsilon >0}
			T(h_i(x),\ell_i(x),\Omega_{\epsilon,j},\Lambda_{\epsilon,j}),
		\end{equation}
	is the localized Sobolev trace constant with  $\Omega_{\epsilon,j} =\Omega \cap B_{\epsilon}(x_j)$ and $\Lambda_{\epsilon,j} =\Omega \cap \partial B_{\epsilon}(x_j)$.		
	\end{theorem}	
	\begin{theorem}[Bonder and Silva  \cite{Bonder}]\label{ccpp}		
			Let $h_i$ and $s_i$ be variable exponents $ \in C_{+}(\overline{\Omega})$
			such that $s_i(x)\leq h_i^{ \ast }(x) , \ \hbox{for all} \ x \in \overline{\Omega}$, 
		and $\{u_m\}_{m\in\mathbb{N}}$ be a sequence in
		$W^{1,h_i(x)}(\Omega)$  such that $u_{i_m} \rightharpoonup u_i$
		weakly  in $W^{1,h_i(x)}(\Omega)$.		
		Then there exist a countable set $J_i^2$, positive numbers
		$ \{\mu_{i_j}\}_{j \in J_i^2}$ and $ \{\nu_{i_j}\}_{j \in J_i^2},$ and 
		$\{x_j\}_{j\in J_i^2}\subset \mathcal{C}_{h_i}^2= \{x\in \Omega\colon s_i(x)=h_i^{\ast}(x)\}$ such that 
		\begin{gather}
		|u_{i_m}|^{s_i(x)}\rightharpoonup	\nu_i=|u_i|^{s_i(x)} + \sum_{j\in I}\nu_{i_j}\delta_{x_j}\ \text{ weakly-*
			in the sense of measures},\\
		|\nabla u_{i_m}|^{h_i(x)} \rightharpoonup	\mu_i \geq |\nabla u_i|^{h_i(x)} + \sum_{j\in I} \mu_{i_j} \delta_{x_j} \ \text{ weakly-* in the sense of measures},\\ \label{2.1b}
			S_i \nu_{i_j}^{1/h_i^\ast(x_j)} \leq \mu_{i_j}^{1/h_i(x_j)} \ \ \hbox{for all} \ j\in J_i^2,
		\end{gather}
		where  
		\begin{equation}\label{GNSS}
			S_i = S_{i_{q_i}}(\Omega) :=\inf_{\phi\in C_0^{\infty}(\Omega)}
			\frac{\| |\nabla v| \|_{L^{p_i(x)}(\Omega)}}{\| v \|_{L^{s_i(x)}(\Omega)}},
		\end{equation}
	is the best constant in the Gagliardo-Nirenberg-Sobolev inequality for variable exponents	
	\end{theorem}	
\textbf{\textbf{Notations.}} Weak (resp. strong) convergence will be denoted by
$\rightharpoonup$
(resp., $\rightarrow$), 
$C_i$, $C_{i_j}$, $c_j$ and $c_{i_j}$ will denote positive constants which may vary from line to line and can be determined in concrete conditions.  Here, $X^{\ast}$  denotes the dual space of $X$, $\delta_{x_j}$ is the Dirac mass at  $x_j$, for all $\rho>0, x\in \Omega$, where $B(x,\rho)$  denotes the ball of radius $\rho$ centered at $x$.	
	\section{The auxiliary problem and variational framework}\label{sec auxiliary}		
	In order to prove Theorem \ref{thm.result}, we shall  introduce the auxiliary problem by defining the auxiliary functional $E_{\theta,\lambda}$ and  showing that the energy functional $E_{\theta,\lambda}$ has the geometry of Mountain  Pass Theorem \ref{PMT}.\\	
	By assumption $(\textbf{\textit{M}})$, we see that the functions $M_i$ are bounded only from below and do not give us enough information about the behavior of $M_i$ at infinity, which makes it difficult to prove that the functional $E_{\lambda}$ has the geometry of Mountain  Pass Theorem and that the sequence of Palais-Smale is bounded in $X$.  Hence,  we truncate functions $M_i$ and study the associated truncated problem.\\	
	 Take  $\gamma_i$ as in assumption ($\textbf{\textit{F}}_3$) and $\theta_i \in \mathbb{R},$ for all $i\in \{1,\dots,N\}$ such that $\mathfrak{M}_i^0 < \theta_i < \frac{\gamma_i\mathfrak{M}_i^0}{p_i^{+}\max\{\beta_{1_i},\beta_{2_i}\}}$. 
	 By assumption $(\textbf{\textit{M}}),$ there exists $\tau_i^0>0$ such that $M_i(\tau_i^0)=\theta_i$. Thus, by setting	
	\[
	M_{\theta_i}(\tau_i)=\begin{cases}
		M_i(\tau_i) &\text{for } 0\leq \tau_i \leq \tau_i^0, \\
		\theta_i &\text{for } \tau_i\geq \tau_i^0,
	\end{cases}
	\]	
 we can introduce the following auxiliary problem	
	\begin{eqnarray}
		\label{s.aux1.1}
		\begin{cases}
			M_{\theta_i}\Big(\mathcal{A}_i(u_i)\Big)
			\Big(-\textrm{div}\,\big(\mathcal{B}_{1_i}(\nabla u_i)\big)+  \mathcal{B}_{2_i}(u_i)\Big)=|u_i|^{s_i (x)-2}u_i+\lambda F_{u_i}(x,u) & \text{in }\Omega, \\
			\mathfrak{N}.M_{\theta_i}\Big(\mathcal{A}_i(u_i)\Big)\mathcal{B}_{1_i}(\nabla u_i)=|u_i|^{\ell_i (x)-2}u_i & \text{on }\partial\Omega ;
		\end{cases}
	\end{eqnarray}	
	for ($ 1\leq i\leq n$), where $\mathcal{A}_i, \mathcal{B}_i, F_{u_i}$ and $\lambda$ are as in Section \ref{s1}. By assumption $(\textbf{\textit{M}})$, we  also know that
	\begin{equation} \label{hypM}
	\mathfrak{M}_i^0 \leq M_{\theta_i}(\tau_i)\leq  \theta_i < \frac{\gamma_i\mathfrak{M}_i^0}{p_i^{+}\max\{\beta_{1_i},\beta_{2_i}\}}, \ \hbox{for all} \ \tau_i\geq 0 ,  1\leq i\leq n.
	\end{equation}	
	Now, the next step is to prove that the auxiliary problem \eqref{s.aux1.1} has a nontrivial weak solution. We obtain the following result.	
	\begin{theorem} \label{AuxTh}
		Suppose that     conditions 
		 $(\textbf{\textit{A}}_1)-(\textbf{\textit{A}}_4)$, $(\textbf{\textit{M}}),$ and $(\textbf{\textit{F}}_1)-(\textbf{\textit{F}}_4)$ hold. Then there exists a constant $\lambda_{\star} >0$, such that if $\lambda\geq \lambda_{\star}$, then problem \eqref{s.aux1.1} has at least one nontrivial solution in $X$.		
	\end{theorem}		
	For the proof of Theorem \ref{AuxTh} we shall need some technical results. We observe that the auxiliary problem \eqref{s.aux1.1} has a variational structure,
	indeed it is the Euler-Lagrange equation of the functional $E_{\theta,\lambda }: X\to \mathbb{R}$ defined as follows	
	\begin{multline*}
		E_{\theta,\lambda }(u)=
		\sum_{i=1}^{n}\widehat{M_{\theta_i}}\left(
		\mathcal{A}_i(u_{i})\right)
		-	\sum_{i=1}^{n}\displaystyle\int_{\Omega  }\dfrac{1}{s_i(x)}|u_i|^{s_i(x)}dx
		-\sum_{i=1}^{n}\displaystyle\int_{\partial\Omega  }\dfrac{1}{\ell_i(x)}|u_i|^{\ell_i(x)}d\sigma_x
		-\int_{\Omega}\lambda F(x,u)dx,
	\end{multline*}
where 
$
\widehat{M}_{\theta_i}(\tau)= \int_{0}^{\tau}M_{\theta_i}(s)ds.
$
Moreover, the functional $E_{\theta,\lambda }$ is Fr\'echet differentiable in $u\in X$ and for all $v=(v_1,\dots,v_n)\in X,$
\begin{equation}\label{derivata}
	\begin{alignedat}2
	\langle E_{\theta,\lambda }'(u),v\rangle &= \sum_{i=1}^{n} M_{\theta_i}\left(
		\mathcal{A}_i(u_{i})\right)\int_{\Omega  } \Big(
		\mathcal{B}_{1_i}(\nabla u_{i}) \nabla v_i  + 	\mathcal{B}_{2_i}(u_{i})v_i 
		\Big) \,dx - \sum_{i=1}^{n}\int_{ \Omega  }|u_{i}|^{s_i(x)-2}u_{i} v_i\,dx \\
		&\quad -\sum_{i=1}^{n}\int_{\partial \Omega  }|u_{i}|^{s_i(x)-2}u_{i} v_i\,d\sigma_x 	- \sum_{i=1}^{n}\displaystyle\int_{\Omega  }\lambda F_{u_i}(x,u)v_i\,dx.
	\end{alignedat}
\end{equation}	
	Now we prove that the functional $E_{\theta,\lambda }$ has the geometric features required by Mountain  Pass Theorem \ref{PMT}.	
	\begin{lemma} \label{lemma32}
		Suppose that conditions  $ (\textbf{\textit{A}}_1)-(\textbf{\textit{A}}_4)$, 
		$(\textbf{\textit{M}})$ and
		$(\textbf{\textit{F}}_1)-(\textbf{\textit{F}}_4)$ hold. Then there exist  positive constants $ \mathcal{R}$ and $\rho$ such that 
		$E_{\theta,\lambda }(u)\geq \mathcal{R}>0,$
		for all $u\in X$ with  $\left\| u\right\|=\rho$.
	\end{lemma}	
	\begin{proof}		
		For all $u=(u_1,\dots,u_n)\in X$, we obtain
		under the assumptions $ (\textbf{\textit{A}}_2)$ and $(\textbf{\textit{A}}_4)$, 		
		\begin{center}
			$\displaystyle \begin{aligned} \sum_{i=1}^n\widehat{M}_{\theta_i}\Big(\mathcal{A}_i(u_i)\Big)& \geq \sum_{i=1}^n \frac{\mathfrak{M}_{i}^0}{p_i^+} \int_{\Omega} \Big(A_{1_i}(|\nabla u_i|^{p_i(x)}) + A_{2_i}(|u_i|^{p_i(x)}) \Big) dx \\
				& \geq \sum_{i=1}^n \frac{\mathfrak{M}_{i}^0}{p_i^+} \int_{\mathbb R^N} \Big[\frac{1}{\beta_{1_i}} a_{1_i}(|\nabla u_i|^{p_i(x)})|\nabla u_i|^{p_i(x)} + \frac{1}{\beta_{2_i}} a_{2_i}(|u_i|^{p_i(x)})|u_i|^{p_i(x)} \Big] dx
				\\ & \geq \sum_{i=1}^n  \frac{\mathfrak{M}_{i}}{p_i^+\max\{\beta_{1_i}, \beta_{2_i}\}} \int_{\mathbb R^N} \Big[a_{1_i}(|\nabla u_i|^{p_i(x)})|\nabla u_i|^{p_i(x)} + a_{2_i}(|u_i|^{p_i(x)})|u_i|^{p_i(x)} \Big] dx \\
				& \geq \sum_{i=1}^n  \frac{\mathfrak{M}_{i}^0\min\{\min\{k_{1_i}^0,k_{2_i}^0\}, \min\{k_{1_i}^2,k_{2_i}^2\}\}}{p_i^+\max\{\beta_{1_i}, \beta_{2_i}\}}  \Big(\rho_{1,p_i(x)}(u_i) + \mathcal{K}(k_{i}^3) \rho_{1,q_i(x)}(u_i) \Big).
			\end{aligned}$
		\end{center}	
		Hence, by using assumption $(\textbf{\textit{F}}_4)$, and relations \eqref{L22}-\eqref{L23}, we have
			\begin{center}
			$\displaystyle \begin{aligned} 	E_{\theta,\lambda }(u) & \geq \sum_{i=1}^n \frac{\mathfrak{M}_{i}^0\min\big\{\min\{k_{1_i}^0,k_{2_i}^0\}, \min\{k_{1_i}^2,k_{2_i}^2\}\big\}}{p_i^+\max\big\{\beta_{1_i}, \beta_{2_i}\big\}}  \Bigg(\min\bigg\{\| u_i\|_{1,p_i(x)}^{p_i^-} , \| u_i\|_{1,p_i(x)}^{p_i^+}\bigg\}\\
			&	+ \mathcal{K}(k_{i}^3)  \min\bigg\{\| u_i\|_{1,q_i(x)}^{q_i^-}, \| u_i\|_{1,q_i(x)}^{q_i^+}\bigg\} \Bigg)
				- \sum_{i=1}^{n}\frac{1}{s_i^{-}} \max\bigg\{ \left\|  u_i\right\|_{L^{s_i(x)}(\Omega)}^{s_i^{-}}, \left\| u_i\right\|_{L^{s_i(x)}(\Omega)}^{s_i^{+}} \bigg\}\\
				&	- \sum_{i=1}^{n}\frac{1}{\ell_i^{-}}\max\bigg\{ \left\|  u_i\right\|_{L^{\ell_i(x)}(\partial\Omega)}^{\ell_i^{-}}, \left\| u_i\right\|_{L^{\ell_i(x)}(\partial \Omega)}^{\ell_i^{+}} \bigg\}
			  -\sum_{i=1}^{n}\lambda  \max\bigg\{ \left\|  u_i\right\|_{L^{w_i(x)}(\Omega)}^{w_i^{-}}, \left\| u_i\right\|_{L^{w_i(x)}(\Omega)}^{w_i^{+}} \bigg\}.
		 \end{aligned}$
		\end{center}	
	By the Sobolev Embedding Theorem and taking $0< \left\| u\right\| =
		\max_{1\leq i\leq n} \big\{\| u_i\|_{1,p_i(x)}+ \mathcal{K}(k_i^3)\|  u_i\|_{1,q_i(x)} \}= \rho <1$, there exist positive constants $c_{1_i}, c_{2_i}, c_{3i}, c_{4i}$ and $c_{5i}$, such that	
		\begin{align*}
			E_{\theta,\lambda }(u) &\geq 	\sum_{i=1}^{n}c_{1_i}\Big(\left\|  u_i\right\|_{1,p_i(x)}^{q^{+}} +\mathcal{K}(k_i^3)  \left\| u_i\right\|_{1,q_i(x)}^{q^{+}} \Big)	- \sum_{i=1}^{n}c_{2_i}\left\|  u_i\right\|_{i}^{s_i^{-}}  	- \sum_{i=1}^{n}c_{3i}\left\|  u_i\right\|_{i}^{\ell_i^{-}}
			-\sum_{i=1}^{n}c_{4i}  \Vert  u_i \Vert^{w_i^-}_{i}\\
			&\geq 	\sum_{i=1}^{n}\Bigl(c_{5i}\left\|  u_i\right\|_{i}^{q_i^{+}}
			-c_{2_i}\left\|  u_i\right\|_{i}^{s_i^{-}}
			-c_{3i}\left\|  u_i\right\|_{i}^{\ell_i^{-}} - c_{4i} \Vert  u_i \Vert^{w_i^-}_{i}
			\Bigr).
		\end{align*}
	Since $q_i^{+}< w_i^{- }< \inf\{s_i^{-},\ell_i^{-}\}$,  there exists a positive constant $\mathcal{R}$ such that
			$
	E_{\theta,\lambda }(u)\geq \mathcal{R}>0, \ \text{ with } \left\| u\right\| =\rho.
		$
	\end{proof} 	
	\begin{lemma} \label{lemma33}
		Suppose that conditions   $(\textbf{\textit{A}}_1)-(\textbf{\textit{A}}_4)$, $(\textbf{\textit{M}}),$ and $(\textbf{\textit{F}}_4)$ hold. Then  for every positive function $\lambda$, there exists a nonnegative function $e\in X$, independent of $\lambda$, such that  $\left\|z\right\|>R$ and $E_{\theta,\lambda }(z)<0$.	
	\end{lemma}	
	\begin{proof} 
		By  the assumptions $(\textbf{\textit{A}}_1)$ and $(\textbf{\textit{A}}_4)$, for all $\delta >0$ and $u\in X$, we have		
		\begin{align*}
			\sum_{i=1}^{n}\widehat{M_{\theta_i}}\left(\mathcal{A}_i(\delta u_i)\right)&=\sum_{i=1}^{n} \int_{0}^{\mathcal{A}_i(\delta u_i)}M_{\theta,i}\left(s\right)ds
			\leq  \sum_{i=1}^{n}\theta_i \mathcal{A}_i(\delta u_i)\\
			&\leq  \sum_{i=1}^{n}\theta_i  \int_{\Omega }\dfrac{1}{p_i(x)}\left(A_{1_i}(|\nabla (\delta u_i)|^{p_i(x)} + A_{2_i}(|\delta u_i|^{p_i(x)})\right)dx\\
			&\leq \sum_{i=1}^{n}
			\theta_i\displaystyle \int_{\Omega }\Big(\dfrac{\max \{k_{1_i}^1,k_{2_i}^1\}}{p_i(x)}\big(|\nabla(\delta  u_i)|^{p_i(x)}
			+|\delta  u_i|^{p_i(x)} \big)
			+
			\dfrac{ k_{i}^3}{q_i(x)} \big(|\nabla(\delta  u_i)|^{q_i(x)} + |\delta  u_i|^{q_i(x)}\big)\Big)dx\\
				&\leq \sum_{i=1}^{n} \theta_i \Big(\dfrac{\max \{k_{1_i}^1,k_{2_i}^1\}}{p_i^-}\max\big\{\| u_i\|_{1,p_i(x)}^{p_i^-} , \| u_i\|_{1,p_i(x)}^{p_i^+}\big\}+ \dfrac{k_{i}^3}{q_i^+} \max\big\{\| u_i\|_{1,q_i(x)}^{q_i^-}, \| u_i\|_{1,q_i(x)}^{q_i^+}\big\} \Big).
		\end{align*} 		
		By this inequality and assumption $(\textbf{\textit{F}}_4)$, we have  for $e=(e_1,\dots,e_n) \in X \setminus \{(0,\dots,0)\big \}$ and each $\delta>1,$
		\begin{align*}
			E_{\theta, \lambda }(\delta e)&=\sum_{i=1}^{n}\widehat{M_i}\left(\mathcal{A}_i(\delta e_i)\right)-\sum_{i=1}^{n}\displaystyle\int_{\Omega  }\dfrac{1}{s_i(x)}|\delta e_i|^{s_i(x)}dx -\sum_{i=1}^{n}\displaystyle\int_{\partial \Omega  }\dfrac{1}{\ell_i(x)}|\delta e_i|^{\ell_i(x)}d\sigma_x- \lambda\int_{\Omega} F(x,\delta e)dx\\
			&\leq \sum_{i=1}^{n}\Bigg[ \theta_i \delta^{q_i^+}\Big(\dfrac{\max \{k_{1_i}^1,k_{2_i}^1\}}{p_i^-}\max\big\{\| e_i\|_{1,p_i(x)}^{p_i^-} , \| e_i\|_{1,p_i(x)}^{p_i^+}\big\}+ \dfrac{k_{i}^3}{q_i^+} \max\big\{\| e_i\|_{1,q_i(x)}^{q_i^-}, \| e_i\|_{1,q_i(x)}^{q_i^+}\big\} \Big)
			\\
				&\quad~~	
			-\frac{\delta^{s_i^-}}{s_i^-}\min\Big\{\| e_i\|_{s_i(x)}^{s_i^-}, \|e_i\|_{s_i(x)}^{s_i^+}\Big\}
			-\frac{\delta^{\ell_i^-}}{\ell_i^-}\min\Big\{\| e_i\|_{\ell_i(x)}^{\ell_i^-}, \|e_i\|_{\ell_i(x)}^{\ell_i^+}\Big\}
			-\delta^{w_i^-}\min\Big\{\| e_i\|_{w_i(x)}^{w_i^-}, \|e_i\|_{w_i(x)}^{w_i^+}\Big\}\Bigg],
		\end{align*}
		which tends to $-\infty$ as $\delta\to +\infty$ since $\min \{s_i^-,\ell_i^-\} >w_i^->q_i^+$. So, the lemma
		is proven by choosing $z=\delta_{*} e$  with $\delta_* >0$ sufficiently large.		
	\end{proof} 	
	Now, by Mountain  Pass Theorem \ref{PMT} without the Palais-Smale condition, we get a sequence $\{u_m\}_{m\in \mathbb{N}}\subset X$ such that	
	$E_{\theta,\lambda}(u_m) \to c_{\theta,\lambda} \text{ and } E_{\theta,\lambda}'(u_m) \to 0,$
	where 
	$
	c_{\theta,\lambda} := \inf_{\phi \in \Gamma} \max_{ \delta \in \left[ 0,1\right]  }E_\lambda(\phi(\delta)),
	$
	and 	
	\[
	\Gamma= \big \{ \phi : \left[ 0,1\right] \to X \  \text{continuous } : \phi(0)=(0,\dots,0), ~~E_{\theta,\lambda}(\phi (1))<0 \big\}.
	\]	
	\begin{lemma} \label{lemma34}		
		Suppose that conditions
		 $(\textbf{\textit{A}}_1)-(\textbf{\textit{A}}_2)$, $(\textbf{\textit{M}}),$ and $(\textbf{\textit{F}}_4)$ hold. Then 
		$ \lim_{\lambda \to +\infty } c_{\theta,\lambda}=0$.
	\end{lemma} 	
	\begin{proof}
		Let $z=(z_1,\dots,z_n)\in X$  be the function given by Lemma \ref{lemma33}. Then $\displaystyle\lim_{\delta \to \infty }E_{\theta,\lambda }(\delta z)=-\infty,$ for each $\lambda >0$, so it follows that there exists $\delta_\lambda>0$ such that
		$E_{\theta,\lambda}(\delta_{\lambda}z)=\max_{\delta\geq 0}E_{\theta,\lambda}(\delta z).$
		Hence,  		
		$
		\langle E'_{\theta,\lambda}(\delta_{\lambda}z), \delta_{\lambda}z\rangle=0,
		$
		so it follows by relation \eqref{derivata}  that		
		\begin{multline}\label{EFD3.3}
			\sum_{i=1}^{n} M_{\theta_i}\left(
			\mathcal{A}_i(\delta_{\lambda }z_{i})\right)\int_{\Omega  } \big(a_{1_i}(|\nabla (\delta_{\lambda }z_{i})| ^{p_i(x)}) |\nabla (\delta _{\lambda }z_{i})| ^{p_i(x)}
			+ a_{2_i}(|\delta_{\lambda }z_{i}| ^{p_i(x)}) | \delta _{\lambda }z_{i}| ^{p_i(x)}\big)\,dx =\sum_{i=1}^{n}\int_{ \Omega  }|\delta_{\lambda }z_{i}|^{s_i(x)}\,dx \\
			+ \sum_{i=1}^{n}\int_{\partial \Omega  }|\delta_{\lambda }z_{i}|^{\ell_i(x)}\,d\sigma_x 	 + \sum_{i=1}^{n}\lambda\delta_{\lambda }\displaystyle\int_{\Omega}F_{u_i}(x, \delta_\lambda z)z_i\,dx.
		\end{multline}		
		By construction, $z_i\geq 0$ a.e. in $\Omega,$ for all $i\in\{1,2,\dots,n\}$. Therefore, by assumption $(\textbf{\textit{F}}_3)$ and relation \eqref{EFD3.3}, 
		\begin{multline}\label{EFD3.4}
			\sum_{i=1}^{n} M_{\theta_i}\left(
		\mathcal{A}_i(\delta_{\lambda }z_{i})\right)\int_{\Omega  } \big(a_{1_i}(|\nabla (\delta_{\lambda }z_{i})| ^{p_i(x)}) |\nabla (\delta _{\lambda }z_{i})| ^{p_i(x)}
		+ a_{2_i}(|\delta_{\lambda }z_{i}| ^{p_i(x)}) | \delta _{\lambda }z_{i}| ^{p_i(x)}\big)\,dx \\\geq \sum_{i=1}^{n}\int_{ \Omega  }|\delta_{\lambda }z_{i}|^{s_i(x)}\,dx 
			+ \sum_{i=1}^{n}\int_{\partial \Omega  }|\delta_{\lambda }z_{i}|^{\ell_i(x)}\,d\sigma_x.
		\end{multline}	
		On the other hand, by assumption $(\textbf{\textit{A}}_2)$ and inequalities \eqref{L22}, \eqref{L23}, 
		\begin{equation}\label{Inq35}
		\begin{alignedat}4
				\sum_{i=1}^{n} M_{\theta_i}\left(
			\mathcal{A}_i(\delta_{\lambda }z_{i})\right)\int_{\Omega  }& \big(	a_{1_i}(|\nabla (\delta_{\lambda }z_{i})| ^{p_i(x)}) |\nabla (\delta _{\lambda }z_{i})| ^{p_i(x)}
			+ a_{2_i}(|\delta_{\lambda }z_{i}| ^{p_i(x)}) | \delta _{\lambda }z_{i}| ^{p_i(x)}
			\big)\,dx\\
			&\leq \sum_{i=1}^{n} \theta_i\Bigg(\int_{\Omega  }\bigg(
			a_{1_i}(|\nabla (\delta_{\lambda }z_{i})| ^{p_i(x)}) |\nabla (\delta _{\lambda }z_{i})| ^{p_i(x)}
			+ a_{2_i}(|\delta_{\lambda }z_{i}| ^{p_i(x)}) | \delta _{\lambda }z_{i}| ^{p_i(x)}
			\bigg)\,dx\Bigg) \\
			&\leq \sum_{i=1}^{n}\theta_i \Big(\max\{k_{1_i}^1,k_{2_i}^1\}\int_{\Omega} \big(|\nabla (\delta_{\lambda }z_{i})| ^{p_i(x)}
			+ |\delta_{\lambda }z_{i}| ^{p_i(x)}\big)\,dx
			+ k_i^3\int_{\Omega} \big(|\nabla (\delta_{\lambda }z_{i})| ^{q_i(x)} + |\delta_{\lambda }z_{i}| ^{q_i(x)}\big)\,dx\Big)\\
			&\leq  \sum_{i=1}^{n}\theta_i \Bigg(\max\{k_{1_i}^1,k_{2_i}^1\}\max \Big\{  \| \delta_{\lambda }z_{i}\|_{1,p_i(x)}^{p_i^-},
			\| \delta_{\lambda }z_{i}\|_{1,p_i(x)}^{p_i^+} \Big\} 
			+k_i^3 \max \Big\{  \|\delta_{\lambda }z_{i}\|_{1,q_i(x)}^{q_i^-},
			\| \delta_{\lambda }z_{i}\|_{1,q_i(x)}^{q_i^+} \Big\} 
			\Bigg).
		\end{alignedat}
	\end{equation}		
		Therefore, from relations \eqref{EFD3.4},\eqref{Inq35}, and inequalities \eqref{L22},\eqref{L23}, we obtain	
		\begin{equation}\label{Inq36}
		\begin{alignedat}4
		 \sum_{i=1}^{n}\theta_i \Bigg(\max\{k_{1_i}^1,k_{2_i}^1\}\max \Big\{  \| \delta_{\lambda }z_{i}\|_{1,p_i(x)}^{p_i^-},&
		\| \delta_{\lambda }z_{i}\|_{1,p_i(x)}^{p_i^+} \Big\} 
		+k_i^3 \max \Big\{  \|\delta_{\lambda }z_{i}\|_{1,q_i(x)}^{q_i^-},
		\| \delta_{\lambda }z_{i}\|_{1,q_i(x)}^{q_i^+} \Big\} 
		\Bigg)  \\
		&\geq \sum_{i=1}^{n}\int_{ \Omega  }|\delta_{\lambda }z_{i}|^{s_i(x)}\,dx 
		+ \sum_{i=1}^{n}\int_{\partial \Omega  }|\delta_{\lambda }z_{i}|^{s_i(x)}\,d\sigma_x
		\\
		&\geq\sum_{i=1}^{n} \min  
		\Big\{  \|\delta_{\lambda }z_{i}\|_{L^{s_i(x)}(\Omega)}^{s_i^-},
		\|\delta_{\lambda }z_{i}\|_{L^{s_i(x)}(\Omega)}^{s_i^+} \Big\}\\
	&\quad	+\sum_{i=1}^{n} \min  
		\Big\{  \|\delta_{\lambda }z_{i}\|_{L^{\ell_i(x)}(\partial\Omega)}^{\ell_i^-},
		\|\delta_{\lambda }z_{i}\|_{L^{\ell_i(x)}(\partial \Omega)}^{\ell_i^+} \Big\}.
		\end{alignedat}
	\end{equation}		
		Next, we shall show that the sequence $\{\delta_\lambda \}$ is bounded in $\mathbb{R}$.  Indeed, we suppose by contradiction that $\{\delta_\lambda \}$  is unbounded. Then there is a subsequence denoted by $\{\delta_{\lambda_m}\},$ with $\delta_{\lambda_m} \to \infty,$ as $m\to +\infty$. Then  by relation \eqref{Inq36},		
	\begin{equation}\label{Inq37}
		\begin{alignedat}4 
			\sum_{i=1}^{n} \Big(\frac{	\theta_i \max\{k_{1_i}^1,k_{2_i}^1\} \|z_{i}\|_{1,p_i(x)}^{p_i^+} }{\delta_{\lambda_m}^{q_{M}^+-p_{i}^+}}	
			+k_i^3 	\theta_i \| z_{i}\|_{1,q_i(x)}^{q_i^+} 
			\Big) 
			&\geq\sum_{i=1}^{n}  \delta_{\lambda }^{s_i^- -q_M^+} \|z_{i}\|_{L^{s_i(x)}(\Omega)}^{s^-}
			+\sum_{i=1}^{n}  \delta_{\lambda }^{\ell_i^- -q_M^+} \|z_{i}\|_{L^{\ell_i(x)}(\partial \Omega)}^{\ell_i^-},
			\end{alignedat}
	\end{equation}
	where $q_{M}^+=\displaystyle\max_{1\leq i\leq n}\{ q_i^+\}$. Therefore, when taking the limit as $m\to +\infty$, we get a contradiction because $p_i^+<q_M^+<\inf\{s_i^-,\ell_i^-\}$. Thus, we can conclude that $\{\delta_\lambda \}$ is indeed bounded in $\mathbb{R}$.\\	
		Consider a sequence $\{\lambda_m\}_{m\in \mathbb{N}}$ such that $\lambda_m\to +\infty $ and let $\delta_0\geq 0$  be such that $ \delta_{\lambda_m} \to \delta_0,$ as $m\to +\infty$. Then by continuity of $M_{\theta_i}$,  $\{M_{\theta_i}(\mathcal{A}_i(\delta_{\lambda_m}z_i))\}_{m\in \mathbb{N}}$ is bounded, for all $i\in \{1,2,\dots,n\}$. Therefore, there exists $C>0$ such that 		
		$$ \sum_{i=1}^{n}M_{\theta_i}(\mathcal{A}_i(\delta_{\lambda_m}z_i))\int_{\Omega  } \big(a_{1_i}(|\nabla \delta_{\lambda_m }z_{i}| ^{p_i(x)}) |\nabla \delta_{\lambda_m }z_{i}| ^{p_i(x)} 
		+ a_{2_i}(|\delta_{\lambda_m }z_{i}| ^{p_i(x)}) | \delta_{\lambda_m }z_{i}| ^{p_i(x)}\big)\,dx \leq C ,
		\text{ for all } m \in \mathbb{N}$$	
			so by inequality \eqref{EFD3.4}, we have		
	\begin{equation}\label{Inq38}
			\sum_{i=1}^{n}\int_{ \Omega  }|\delta_{\lambda_m }z_{i}|^{s_i(x)}\,dx 
		+ \sum_{i=1}^{n}\int_{\partial \Omega  }|\delta_{\lambda_m }z_{i}|^{\ell_i(x)}\,d\sigma_x + \sum_{i=1}^{n}\displaystyle\int_{\Omega  }\lambda \delta_{\lambda_m }F_{u_i}(x, \delta_{\lambda_m} z)z_i\,dx \leq C,
		\text{ for all } m \in \mathbb{N}.
	\end{equation}
		We shall prove that $\delta_0=0$. Indeed, if $\delta_0>0$, then by assumption $(\textbf{\textit{F}}_2)$, there exist positive functions $b_{i_j}$ ($1\leq i,j \leq n$), such that
		\begin{gather*}
			\Big|  F_{\xi_i}(x,\xi_1,\dots,\xi_n)\Big| \leq
			\sum_{j=1}^{n} b_{i_j}(x)| \xi_j|^{\ell_{i_j}-1},
				\ \hbox{	where} \ 
			1<\ell_{i_j}<\inf_{x\in \Omega}h_i(x),
				\ \hbox{ for all} \  
				x\in \Omega.
		\end{gather*}	
	Thus, by the Lebesgue Dominated Convergence Theorem, we get 		
		$$\sum_{i=1}^{n}\displaystyle\int_{\Omega  }\lambda \delta_{\lambda_m }F_{u_i}(x, \delta_{\lambda_m} z)z_i\,dx 
		\to \sum_{i=1}^{n}\displaystyle\int_{\Omega  }\lambda \delta_{0 }F_{u_i}(x, \delta_0 z)z_i\,dx ,
		\ \hbox{ as } \ 
		m\to +\infty$$	
	By remembering that $\lambda_m\to +\infty$, we find		
		\begin{equation*}
			\sum_{i=1}^{n}\int_{ \Omega  }|\delta_{\lambda_m }z_{i}|^{s_i(x)}\,dx 
			+ \sum_{i=1}^{n}\int_{\partial \Omega  }|\delta_{\lambda_m }z_{i}|^{\ell_i(x)}\,d\sigma_x + \sum_{i=1}^{n}\displaystyle\int_{\Omega  }\lambda \delta_{\lambda_m }F_{u_i}(x, \delta_{\lambda_m} z)z_i\,dx \to +\infty,
			\text{ as } \lambda_m \to +\infty.
		\end{equation*}		
	This contradicts the fact \eqref{Inq38}, so we can deduce that $\delta_0=0$.\\		
	 	Next, we consider the following path $\phi_*(\delta)=\delta z$ for $\delta\in[0,1]$ which belongs to $\Gamma$. By using assumption $(\textbf{\textit{F}}_3)$, we obtain
		\begin{equation}\label{inq39}
			0<c_{\theta,\,\lambda_m}\leq\max_{\delta\in[0,1]}E_{\theta,\,\lambda_m}(\phi_*(\delta))\leq E_{\theta,\,\lambda}(\delta_{\lambda_m} z)\leq 	\sum_{i=1}^{n}\widehat{M_{\theta_i}}\left(\mathcal{A}_i(\delta_{\lambda_m}z_i)\right).
		\end{equation}		
	On the other hand, since $M_{\theta_i}$ are continuous 
	for all $1\leq i \leq n,$ and $\delta_0=0$, we get
		$\lim_{ m\to +\infty } \widehat{M_{\theta_i}}\left(\mathcal{A}_i(\delta_{\lambda_m}z_i)\right)=0 ,$
		for all
		$ 
		i\in \{1,2,\dots,n\}.$
	Thus, from relation \eqref{inq39}, we get 
		$\lim_{ m\to +\infty } c_{\theta,\,\lambda_m}=0.$
	Moreover, by using also assumption $(\textbf{\textit{F}}_3),$ we achieve that the sequence $\{c_{\theta,\,\lambda}\}_{\lambda}$ is monotone. Therefore, we have completed the proof.
	\end{proof}
	\begin{lemma}\label{lemma35}
		If $\left\lbrace u_m=(u_{1_m},u_{2_m},\dots,u_{n_m})\right\rbrace_{m\in \mathbb{N}} $ is a Palais-Smale sequence for $E_{\theta,\lambda }$, then $\left\lbrace u_m\right\rbrace_{m}$ is bounded in  $X$.
	\end{lemma}	
	\begin{proof}
		Let  $\left\lbrace u_m=(u_{1_m},u_{2_m},\dots,u_{n_m})\right\rbrace_{m} $  be a $(PS)_c$ for $E_{\theta,\lambda }$. Then
		we have
		\begin{align*}
			E_{\theta,\lambda }(u_m)&=\sum_{i=1}^{n}\widehat{M_{\theta_i}}\left(
			\mathcal{A}_i(u_{i_m}(x))\right)-
			\sum_{i=1}^{n}\displaystyle\int_{ \Omega  }\dfrac{1}{s_i(x)}|u_{i_m}(x)|^{s_i(x)}dx -
			\sum_{i=1}^{n}\displaystyle\int_{\partial \Omega  }\dfrac{1}{\ell_i(x)}|u_{i_m}(x)|^{\ell_i(x)}d\sigma_x \\
			&\quad - \int_{\Omega}\lambda F(x,u_m)dx
			= c+o_m(1).
		\end{align*}
		On the other hand, for all $v=(v_1,v_2,\dots,v_n)\in  X$, we have
		\begin{multline}\label{e3.1}
			\langle E_{\theta,\lambda }'(u_m),v\rangle = \sum_{i=1}^{n} M_{\theta_i}\left(
			\mathcal{A}_i(u_{i_m})\right)\int_{\Omega  } \bigg(
			\mathcal{B}_{1_i}(\nabla u_{i_m}) \nabla u_{i_m} \nabla v_i
			+ \mathcal{B}_{2_i}(u_{i_m}) u_{i_m}  v_i\bigg)\,dx
			- \sum_{i=1}^{n}\int_{ \Omega  }|u_{i_m}|^{s_i(x)-2}u_{i_m} v_i\,dx \\
			\quad - \sum_{i=1}^{n}\int_{\partial \Omega  }|u_{i_m}|^{s_i(x)-2}u_{i_m} v_i\,d\sigma_x 	- \sum_{i=1}^{n}\displaystyle\int_{\Omega  }\lambda F_{u_i}(x,u_m)v_i\,dx
			=o_m(1).
		\end{multline}
		Thus,
		\begin{align*}
			 E_{\theta, \lambda }(u_m)-\langle	E_{\theta,\lambda }'(u_m),\frac{u_m}{\gamma }\rangle&\geq \sum_{i=1}^{n} \Bigg(
			\widehat{M_{\theta_i}}\left(
			\mathcal{A}_i(u_{i_m})\right) -
			\frac{1}{\gamma_i}M_{\theta_i}\left(
			\mathcal{A}_i(u_{i_m})\right)\int_{\Omega  } \bigg(a_{1_i}(|\nabla u_{i_m}| ^{p_i(x)}) |\nabla u_{i_m}| ^{p_i(x)} \\
			&+a_{2_i}(|u_{i_m}| ^{p_i(x)}) | u_{i_m}| ^{p_i(x)} \bigg)\,dx
			\Bigg)
			+ \sum_{i=1}^{n} \Big(\frac{ 1}{\gamma_i}-\frac{ 1}{s_i^-}\Big)\displaystyle\int_{\Omega  }|u_{i_m}|^{s_i(x)}dx\\
			&	+ \sum_{i=1}^{n} \Big(\frac{ 1}{\gamma_i}-\frac{ 1}{\ell_i^-}\Big)\displaystyle\int_{\partial\Omega  }|u_{i_m}|^{\ell_i(x)}d\sigma_x  +\lambda  \displaystyle\int_{\Omega}
			\left[ \sum_{i=1}^{n}  \frac{u_{i_m}}{\gamma_i}F_{ u_i}(x,u_{m})	-F(x,u_{m})
			\right] dx, 
		\end{align*}
		where $\frac{u_m}{\gamma}=(\frac{u_{1_m}}{\gamma_1},\frac{u_{2_m}}{\gamma_2},\dots,\frac{u_{n_m}}{\gamma_n})$.
		Therefore, by using assumptions $(\textbf{\textit{A}}_4)$, $(\textbf{\textit{M}})$ and  $(\textbf{\textit{F}}_3)$, we
		can conclude that		
		\begin{align*}
		E_{\theta,\lambda }(u_m)- \langle	E_{\theta,\lambda }'(u_m),\frac{u_m}{\gamma }\rangle	&\geq \sum_{i=1}^{n} \mathfrak{M}_i^0\Bigg(\frac{\theta_i}{p_i^+}\int_{\Omega}\big(A_{1_i}(|\nabla u_{i_m}| ^{p_i(x)}) + A_{2_i}(| u_{i_m}| ^{p_i(x)})\big)dx\\
			&\quad-
			\frac{1}{\gamma_i }\int_{\Omega}\big(a_{1_i}(|\nabla u_{i_m}|^{p_i(x)})|\nabla u_{i_m}|^{p_i(x)} + a_{2_i}(| u_{i_m}|^{p_i(x)})| u_{i_m}|^{p_i(x)}\big)dx
			\Bigg)\\
			&\geq \sum_{i=1}^{n}\Big(\frac{\theta_i\mathfrak{M}_i^0}{p_i^+  \max\{\beta_{1_i}, \beta_{2_i}\}}- \frac{\mathfrak{M}_i^0}{\gamma_i }\Big)\int_{\Omega}\bigg(a_{1_i}(|\nabla u_{i_m}|^{p_i(x)})|\nabla u_{i_m}|^{p_i(x)}+ a_{2_i}(|u_{i_m}|^{p_i(x)})| u_{i_m}|^{p_i(x)}\bigg)dx.
		\end{align*}	
		By using assumption $(\textbf{\textit{A}}_4)$, we can find positive constants $C_{i1}$ and $C_{i2}$ such that		
		\begin{align}\label{key}
		E_{\theta,\lambda }(u_m)-\langle	E_{\theta,\lambda }'(u_m),\frac{u_m}{\gamma }\rangle	&\geq \sum_{i=1}^{n} \Bigg[ C_{i1} \Big(\int_{\Omega}\big(|\nabla u_{i_m}|^{p_i(x)}+ | u_{i_m}|^{p_i(x)}\big)dx\Big) +C_{i2}\mathcal{K}(k_i^3)\Big(\int_{\Omega}\big(|\nabla u_{i_m}|^{q_i(x)}+| u_{i_m}|^{q_i(x)}\big)dx\Big)\Bigg].
		\end{align}
	To prove the assertion, we assume by contradiction that
		$\left\|  u_{i_m}\right\|_i =\| u_i\|_{1,p_i(x)}+ \mathcal{K}(k_i^3)\|  u_i\|_{1,q_i(x)}
		\to + \infty $.		
	So, if $k_i^3=0$, then  by using relation \eqref{L22}, we have 
		$
		E_{\theta,\lambda }(u_m)-\langle	E_{\theta,\lambda }'(u_m),\frac{u_m}{\gamma }\rangle 
		\geq  \sum_{i=1}^{n} C_{i}  \| u_{i_m}\|_{i}^{p_i^{-}}.
		$	
		Thus, we
		can
		 find
		$
			c+ o_m(1)
			\geq \sum_{i=1}^{n} C_{i}  \| u_{i_m}\|_{i}^{p_i^{-}}.
		$
	Since $p_i^->1$, we obtain a contradiction.  Hence, we can deduce that $\left\lbrace u_{m}\right\rbrace $ is bounded in $X$.		\\
		When $k_i^3>0$,  we have three cases to analyze.\par
		
		$(i)  \left\|  u_{i_m} \right\| _{1,p_i(x)} \to +\infty$ and $\left\|  u_{i_m}\right\| _{1,q_i(x)} \to +\infty,$  as $m \to  +\infty$,
		
		$(ii)  \left\|  u_{i_m} \right\| _{1,p_i(x)} \to +\infty$  and $\left\|  u_{i_m}\right\| _{1,q_i(x)}$ is bounded,		
		
		$(iii)  \left\|  u_{i_m} \right\| _{1,p_i(x)}$ is bounded and $\left\|  u_{i_m}\right\| _{1,q_i(x)}\to +\infty$.\\
		In the  case $(i)$, for $m$ sufficiently large, we have $\| u_{i_m}\|_{1,q_i(x)}^{q^-} \geq \| u_{i_m}\|_{1,q_i(x)}^{p^-} $. Hence, by inequality \eqref{key}, we get		
		\begin{align*}
			c+ o_m(1)&\geq \sum_{i=1}^{n}  \Big[ C_{1_i} \| u_{i_m}\|_{1,p_i(x)}^{p_i^{-}} +C_{2_i}\mathcal{K}(k_i^3)\| u_{i_m}\|_{1,q_i(x)}^{q_i^{-}}\Big]\\
			&\geq \sum_{i=1}^{n}   \Big[ C_{1_i} \| u_{i_m}\|_{1,p_i(x)}^{p_i^{-}} +C_{2_i}\mathcal{K}(k_i^3)\| u_{i_m}\|_{1,q_i(x)}^{p_i^{-}}\Big]\\
			&\geq \sum_{i=1}^{n} C_{3i} \| u_{i_m}\|_{_i}^{p_i^{-}},
		\end{align*}
		and this is a contradiction.		
		In the case $(ii)$, by using  inequality \eqref{key}, we conclude that
		$
			c+ o_m(1)
			\geq \sum_{i=1}^{n} C_{1_i} \| u_{i_m}\|_{1,p_i(x)}^{p_i^{-}} ,
		$	
	Hence,  we also get a contradiction when  limit as $m\to +\infty$ because $p_i^->1$.
		In the  case $(iii)$, the proof is similar as in the case $(ii)$ so we shall omit it.			Finally, we can deduce that $\{u_{m}\} $ is bounded sequence in $X$.		
	\end{proof}
Next, we shall prove that the auxiliary problem \eqref{s.aux1.1} possesses at least one nontrivial weak solution.	
	\begin{proof}[\bf Proof of Theorem \ref{AuxTh}]
		By Lemmas \ref{lemma32} and \ref{lemma33}, the functional $ E_{\theta,\,\lambda}$ satisfies the geometric structure required by Mountain  Pass Theorem \ref{PMT}. Now, it remains to check the validity of the Palais-Smale condition. Let $\left\{u_m=(u_{1_m},u_{2m},\dots,u_{n_m})\right\}_{m\in\mathbb{N}}$ be a Palais-Smale sequence at the level $c_{\theta,\,\lambda}$  in $X$. Then  Lemma \ref{lemma34} implies that there exists $\lambda_{*}$ such that		
			$$c_{\theta,\lambda }< \min \left\lbrace
		 \min_{1\leq i\leq n} \left\lbrace \Big(\frac{1}{\gamma_i}-\frac{1}{{\ell_i}_{\mathcal{C}^1_{h_{i}}}^-}\Big)\inf_{j\in J_i^1}\Big\{\overline{T_{i}}_{x_j}^N\Big(D_i\Big)^{N/h_i(x_j)}\Big\} \right\rbrace
		  ,
		 \min_{1\leq i\leq n} \left\lbrace \Big(\frac{1}{\gamma_i}-\frac{1}{{s_i}_{\mathcal{C}^2_{h_{i}}}^-}\Big)\inf_{j\in J_i^2}\Big\{S_{i}^N\Big(D_i\Big)^{N/h_i(x_j)}\Big\} \right\rbrace
		   \right\rbrace , $$
		where $\overline{T_{i}}_{x_j}$ and $S_{i } $ are given respectively  in relations \eqref{GNS} and \eqref{GNSS} and 
		$$D_i= \mathfrak{M}_i^0\big(\min\{k_{1_i}^0, k_{2_i}^0\}(1-\mathcal{K}(k_i^3)+\mathcal{K}(k_i^3)\min\{k_{1_i}^2,k_{2_i}^2\})\big).$$ 	
		So, there exists a subsequence strongly convergent in $X$. Indeed, applying  Lemma \ref{lemma35}, $\{ u_m\}_{m\in \mathbb{N}}$ is bounded in $X$, passing to a subsequence, still denoted by $\{ u_m\}_{m}$ weakly convergent in $X$, so there exist positive  bounded measures $\mu_i$, $\nu_i \in  \Omega $ and $\overline{\nu}_i \in \partial \Omega $ such that		
		$
		|\nabla u_{i_m}|^{h_i(x)}\rightharpoonup \mu_i , \quad
		| u_{i_m}|^{s_i(x)}\rightharpoonup \nu_i ,~ \text{    and   } ~| u_{i_m}|^{\ell_i(x)}\rightharpoonup \overline{\nu_i} .
		$
		Hence, by Theorems \ref{ccp} and \ref{ccpp}, if $\bigcup_{i=1}^{n}(J_i^1\cup J_i^2) =\emptyset ,$ then $u_{i_m} \rightharpoonup u_i$ in  $L^{s_i(x)}(\Omega)$ and $u_{i_m} \rightharpoonup u_i$ in  $ L^{\ell_i(x)}(\partial \Omega),$ for all $1\leq i\leq n$. Let us prove that if
		$$c_{\theta,\lambda}< \min \left\lbrace
		\min_{1\leq i\leq n} \left\lbrace \Big(\frac{1}{\gamma_i}-\frac{1}{{\ell_i}_{\mathcal{C}^1_{h_{i}}}^-}\Big)\inf_{j\in J_i^1}\Big\{\overline{T_{i}}_{x_j}^N\Big(D_i\Big)^{N/h_i(x_j)}\Big\} \right\rbrace,
		\min_{1\leq i\leq n} \left\lbrace \Big(\frac{1}{\gamma_i}-\frac{1}{{s_i}_{\mathcal{C}^2_{h_{i}}}^-}\Big)\inf_{j\in J_i^2}\Big\{S_{i}^N\Big(D_i\Big)^{N/h_i(x_j)}\Big\} \right\rbrace
		\right\rbrace$$
		 and $\{u_m\}_{m\in \mathbb{N}} $ is a Palais-Smale sequence with energy level $c_{\theta,\lambda},$ then $J_i^1\cup J_i^2=\emptyset,$ for all  $1\leq i\leq n$. In fact, suppose there is $i$ in $\{1,\dots,n\} $such that  $J_i^1\cup J_i^2$ is nonempty, then $J_i^1\neq\emptyset$ or $J_i^2\neq\emptyset$.\\
		First, we assume the case $J_i^1\neq\emptyset$.
		Let $x_j\in \mathcal{C}^1_{h_{i}}$ be a singular point of the measures $\mu_i$ and $\overline{\nu_i}$. 
		Consider  $\psi \in C^{\infty}_0(\mathbb{R}^N)$, such that
		$0\leq \psi(x) \leq 1$,  $\psi(0)=1,$ and $\text{supp}\psi\subset B(0,1)$.		
	We consider, for each $j\in J_i^1$ and any $\varepsilon>0$, the functions			
		$ 	 \psi_{j,\varepsilon} := \psi \Big(\frac{x-x_j}{\varepsilon}\Big),$
		for all
		$ x \in  \mathbb{R}^N.$		
		Notice that $\psi_{j,\varepsilon} \in C^{\infty}_0(\mathbb{R}^N, \left[ 0,1\right] )$, $|\nabla \psi_{j,\epsilon}|_\infty\leq \frac{2}{\varepsilon},$ and
		\begin{eqnarray*}
			\psi_{j,\varepsilon}(x)=
			\begin{cases}
				1,  &  x\in  B_{}(x_j,\varepsilon),\\
				
				0, 	&  x\in \mathbb{R}^N\setminus B_{}(x_j,2\varepsilon). \\
			\end{cases}
		\end{eqnarray*}
		Since  $\{u_{i_m}\}_{m}$ is bounded in  $W^{1,h_i(x)}(\Omega)\cap W^{1,p_i(x)}(\Omega)$, the sequence $ \left\lbrace u_{i_m}\psi_{j,\varepsilon}\right\rbrace $ is also bounded in $W^{1,h_i(x)}(\Omega)\cap W^{1,p_i(x)}(\Omega)$. So, by relation \eqref{e3.1}, we obtain 
	$
\langle E_{\theta,\lambda }'(u_{1_m},\dots, u_{i_m},\dots, u_{n_m}),(0,\dots,u_{i_m}\psi_{j,\varepsilon},\dots,0
		)\rangle \rightarrow 0 \text{ as }  m \rightarrow +\infty.
			$
		Therefore, we have 
		\begin{align*}
			\langle E_{\theta, \lambda }'(u_m)(0,\dots, u_{i_m}\psi_{j,\varepsilon},\dots,0
			)\rangle &=   M_{\theta_i}\left(
			\mathcal{A}_i(u_{i_m})\right)\int_{\Omega  }\bigg(a_{1_i}(|\nabla u_{i_m}| ^{p_i(x)}) |\nabla u_{i_m}| ^{p_i(x)-2}\nabla u_{i_m}\nabla (u_{i_m}\psi_{j,\varepsilon}) \\
		&	 +
			a_{2_i}(|u_{i_m}| ^{p_i(x)}) | u_{i_m}| ^{p_i(x)-2}u_{i_m} (u_{i_m}\psi_{j,\varepsilon})
			\bigg)\,dx
				- \int_{\Omega  }|u_{i_m}|^{s_i(x)-2}u_{i_m} (u_{i_m}\psi_{j,\epsilon})\,dx\\
		&	- \int_{\partial \Omega  }|u_{i_m}|^{\ell_i(x)-2}u_{i_m} (u_{i_m}\psi_{j,\epsilon})\,d\sigma_x		
				-\int_{\Omega  }\lambda  F_{u_i}(x,u_{m})u_{i_m}\psi_{j,\varepsilon}\,dx   \rightarrow 0 \text{ as } m \rightarrow +\infty
		\end{align*}		
		and we
		can
		 find
		\begin{multline}\label{3.2}
			M_{\theta_i}\left(
			\mathcal{A}_i(u_{i_m})\right)\int_{\Omega  } a_{1_i}(|\nabla u_{i_m}| ^{p_i(x)}) |\nabla u_{i_m}| ^{p_i(x)-2}\nabla u_{i_m}\nabla \psi_{j,\varepsilon} u_{i_m} \,dx = \int_{\Omega  }|u_{i_m}|^{s_i(x)}\psi_{j,\epsilon}\,dx
			+\int_{\partial \Omega  }|u_{i_m}|^{\ell_i(x)}\psi_{j,\epsilon}\,d\sigma_x
			\\-M_{\theta_i}\left(
			\mathcal{A}_i(u_{i_m})\right)	
			\int_{\Omega  } \bigg(a_{1_i}(|\nabla u_{i_m}| ^{p_i(x)}) |\nabla u_{i_m}| ^{p_i(x)} 
			+ a_{2_i}(| u_{i_m}| ^{p_i(x)}) |u_{i_m}| ^{p_i(x)}
			\Big)\psi_{j,\varepsilon} \,dx  + \int_{\Omega  }\lambda F_{u_i}(x,u_{m})u_{i_m}\psi_{j,\varepsilon}\,dx +o_m(1).
		\end{multline}		
		Next, we shall prove that		
		\begin{align}\label{3.3}
			\lim_{\varepsilon \to 0} \left\lbrace
			\limsup_{m\to +\infty } M_{\theta_i}\left(
			\mathcal{A}_i(u_{i_m})\right)\int_{\Omega  } a_i(|\nabla u_{i_m}| ^{p_i(x)}) |\nabla u_{i_m}| ^{p_i(x)-2}\nabla u_{i_m}\nabla\psi_{j,\varepsilon} u_{i_m} \,dx\right\rbrace  	& = 0.
		\end{align}		
		Notice that, due to assumption $(\textbf{\textit{A}}_2),$ it suffices to show that		
		\begin{align}\label{3.4}
			\lim_{\varepsilon \to 0} \left\lbrace
			\limsup_{m\to +\infty } M_{\theta_i}\left(
			\mathcal{A}_i(u_{i_m})\right)\int_{\Omega  } |\nabla u_{i_m}| ^{p_i(x)-2}\nabla u_{i_m}\nabla\psi_{j,\varepsilon} u_{i_m} \,dx\right\rbrace  & = 0.
		\end{align}
		and		
		\begin{align}\label{3.5}
			\lim_{\varepsilon \to 0} \left\lbrace
			\limsup_{m\to +\infty } M_{\theta_i}\left(
			\mathcal{A}_i(u_{i_m})\right)\int_{\Omega  } |\nabla u_{i_m}| ^{q_i(x)-2}\nabla u_{i_m}\nabla\psi_{j,\varepsilon} u_{i_m} \,dx\right\rbrace  & = 0.
		\end{align}		
		First, by applying  H\"{o}lder inequality, we obtain
		\begin{align*}
			\left| \int_{\Omega  }  |\nabla u_{i_m}| ^{p_i(x)-2}\nabla u_{i_m}\nabla\psi_{j,\varepsilon} u_{i_m}\,dx\right|
			&	\leq 2\left\|  \left|\nabla u_{i_m}\right|^{p_i(x)-1}\right\| _{L^{\frac{p_i(x)}{p_i(x)-1}}(\Omega)}\left\| \nabla\psi_{j,\varepsilon} u_{i_m}\right\| _{L^{p_i(x)}(\Omega)},
		\end{align*}
		since $\left\lbrace u_{i_m}\right\rbrace $ is bounded, the real-valued sequence $ \left\|  \left|\nabla u_{i_m}\right|^{p_i(x)-1}\right\| _{L^{\frac{p_i(x)}{p_i(x)-1}}(\Omega)}$ is also bounded, thus there exists a positive constant $C_i$ such that		
		\begin{align*}
			\left| \int_{\Omega  }  |\nabla  u_{i_m}| ^{p_i(x)-2}\nabla  u_{i_m}\nabla\psi_{j,\varepsilon}  u_{i_m} \,dx\right|
			&	\leq  C_i\left\|\nabla\psi_{j,\varepsilon}  u_{i_m}\right\|_{L^{p_i(x)}(\Omega)}.
		\end{align*}		
		Moreover, the sequence $\{ u_{i_m}\}$ is bounded in $W^{1,p_i(x)}(B_{}(x_j,2\varepsilon))$, so
		there is a subsequence, again denoted by $\left\lbrace  u_{i_m}\right\rbrace ,$  converging weakly to $u_i$ in $L^{p_i(x)}(B_{}(x_j,2\varepsilon))$. Therefore, 	
		\begin{align*}
			\limsup_{m \to +\infty }	\left| \int_{\Omega  }  |\nabla u_{i_m}| ^{p_i(x)-2}\nabla u_{i_m}\nabla\psi_{j,\varepsilon} u_{i_m}\,dx\right|
			&	\leq  C_i\left\|\nabla\psi_{j,\varepsilon}  u_{i}\right\|_{L^{p_i(x)}(\Omega)}\\
			&	\leq 2C_i \limsup_{\varepsilon \to 0}
			\big\||\nabla\psi_{j,\varepsilon}|^{p_i(x)}\big\|_{L^{\big(\frac{p_i^{\ast}(x)}{p_i(x)}\big)^{'}}(B_{}(x_j,2\varepsilon))} \big\||u_i|^{p_i(x)}\big\|_{L^{\frac{p_i^{\ast}(x)}{p_i(x)}}(B_{}(x_j,2\varepsilon))}	\\
			&	\leq 2C_i	\limsup_{\varepsilon \to 0 } \big\||\nabla\psi_{j,\varepsilon}|^{p_i(x)}\big\|_{L^{\frac{N}{p_i(x)}}(B_{}(x_j,2\varepsilon))} \big\||u_i|^{p_i(x)}\big\|_{L^{\frac{N}{N-p_i(x)}}(B_{}(x_j,2\varepsilon))}.
		\end{align*}		
		Note that		
		\[
		\int_{B_{}(x_j,2\varepsilon) }
		(|\nabla\psi_{j,\varepsilon}|^{p_i(x)})^{\frac{N}{p_i(x)}}dx = \int_{B_{}(x_j,2\varepsilon) } |\nabla\psi_{j,\varepsilon}|^{N}dx
		\leq \Big(\frac{2}{\varepsilon}\Big)^N \hbox{meas} (B_{}(x_j,2\varepsilon))=\frac{4^N}{N}\omega_N,
		\]
		where $\omega_N$ is the surface area of the $N$-dimensional unit sphere. Since $\displaystyle \int_{B_{}(x_j,2\varepsilon) } (|u_i|^{p_i(x)})^{\frac{N}{N-p_i(x)}}dx \to 0,$ when $\varepsilon \to 0$, we obtain that $\left\|\nabla\psi_{j,\varepsilon}  u_{i}\right\|_{L^{p_i(x)}(\Omega)}\to 0$, which implies
		\begin{equation}\label{con}
			\lim_{\varepsilon \to 0} \left\lbrace
			\limsup_{n\to +\infty }\left|
			\int_{\Omega  }  |\nabla u_{i_m}| ^{p_i(x)-2}\nabla u_{i_m}\nabla\psi_{j,\epsilon} u_{i_m} \,dx \right| \right\rbrace = 0.
		\end{equation}		
		Since the sequence $\left\lbrace  u_{i_m}\right\rbrace $  is bounded in $W^{1,h_i(x)}(\Omega)\cap W^{1,p_i(x)}(\Omega) $, we may assume that  $\mathcal{A}_i(u_{i_m}) \to \xi_i\geq 0 ,$ as $ m \to +\infty$. Note that $M_i(\xi_i)$ is is continuous, so
		we  have
				$
		M_i\Big(\mathcal{A}_i(u_{i_m})\Big) \to M_i(\xi_i)\geq \mathfrak{M}_i^0>0, $
		as
		$ m \to +\infty.
		$
		Therefore, by relation \eqref{con}, we obtain
		\begin{equation}\label{3.7}
			\lim_{\varepsilon \to 0} \left\lbrace
			\limsup_{m\to +\infty }M_i\left(\mathcal{A}_i(u_{i_m})\right)
			\int_{\Omega  }  |\nabla u_{i_m}| ^{p_i(x)-2}\nabla u_{i_m}\nabla\psi_{j,\varepsilon} u_{i_m} \,dx\right\rbrace   = 0.
		\end{equation}		
		Analogously, we can verify relation \eqref{3.5}. Hence, we have completed the proof
		of relation
		\eqref{3.3}.
		Similarly, we can also obtain		
		\begin{equation}\label{3.8}
			\lim_{\varepsilon \to 0 }
			\int_{\Omega  }\lambda  F_{u_i}(x,u_{m})\psi_{j,\epsilon} u_{i_m} dx=0, \text{ as } m \rightarrow +\infty.
		\end{equation}		
		By applying  H\"{o}lder inequality, assumption $(\textbf{\textit{F}}_2)$ and the fact
		that $0\leq \psi_{j,\varepsilon} \leq 1$,  we have		
		\begin{align*}
			\lim_{\varepsilon \to 0 }
			\int_{\Omega  }\lambda F_{ u_i}(x,u_{m})\psi_{j,\varepsilon} u_{i_m} d x &\leq
			\lim_{\varepsilon \to 0 }\lambda \int_{\Omega  }	\Big(\sum_{j=1}^{n} b_{i_j}(x)| u_jm|^{r_{i_j}-1}\Big) \psi_{j,\varepsilon} u_{i_m}dx\\
			&\leq \lim_{\varepsilon \to 0 }c\int_{\Omega  }	\Big(\sum_{j=1}^{n} b_{i_j}(x)| u_j|^{r_{i_j}-1}\Big)|\psi_{j,\varepsilon} u_{i_m}|dx\\
			&\leq \lim_{\varepsilon \to 0 }c_1 \Big(\sum_{j=1}^n
			|b_{i_j}|_{\alpha_{i_j}(x)}\big\||u_{jm}|^{r_{i_j}-1}\big\|_{L^{q_j^{\ast}(x)}(\Omega)}\big\|\psi_{j,\epsilon} u_{i_m}\big\|_{L^{q_i^{\ast}(x)}(\Omega)}
			\Big).
		\end{align*}		
		This  yields
		\begin{align*}
			\lim_{\varepsilon \to 0 }
			\int_{\Omega  }\lambda F_{ u_i}(x,u_{m})\psi_{j,\varepsilon} u_{i_m} d x &\leq \lim_{\varepsilon \to 0 }c_1
			\Big(\sum_{j=1}^n
			\|b_{i_j}\|_{L^{\alpha_{i_j}(x)}(B_{}(x_j,2\varepsilon))}\|u_{jm}\|_{L^{q_{j}(x)}(B_{}(x_j,2\varepsilon))}^{r_{i_j}-1}\Big)\big\| u_{i_m}\big\|_{L^{q_i(x)}(B_{}(x_j,2\varepsilon))},
		\end{align*}
		and the last term on the right-hand goes to zero, because		
		$\sum_{j=1}^n
		\|b_{i_j}\|_{L^{\alpha_{i_j}(x)}(B_{}(x_j,2\varepsilon))}\|u_j\|_{L^{q_{j}(x)}(B_{}(x_j,2\varepsilon))}^{r_{i_j}-1}<\infty.$
		Therefore, we have completed the proof of relation \eqref{3.8}.
		On the other hand, we have		
		$ \lim_{\varepsilon \to 0}\int_{\Omega} \psi_{j,\epsilon}d\mu_{i_j} =\mu_{i_j} \psi(0) $
		and 
		$
		\lim_{\varepsilon \to 0}\int_{\partial \Omega} \psi_{j,\epsilon}d\overline{\nu}_{i_j} =\overline{\nu}_{i_j} \psi(0),$
	and since $\mathcal{C}_{h_i}^1 \cap \mathcal{C}_{h_i}^2=\emptyset$, for $\epsilon>0$ sufficiently small, we have 
		$$
		\int_{\Omega  }|u_{i_m}| ^{p_i(x)}\psi_{j,\varepsilon}dx \to \int_{\Omega  }|u_{i}| ^{p_i(x)}\psi_{j,\varepsilon}dx
			~~\text{, }~~
		\int_{\Omega  }|u_{i_m}| ^{q_i(x)}\psi_{j,\varepsilon}dx \to \int_{\Omega  }|u_{i}| ^{q_i(x)}\psi_{j,\varepsilon}dx,
		\
		\int_{\Omega  }|u_{i_m}| ^{s_i(x)}\psi_{j,\varepsilon}dx \to \int_{\Omega  }|u_{i}| ^{s_i(x)}\psi_{j,\varepsilon}dx,
		$$
		hence when  $\epsilon\to 0$,		
		$$
		\int_{\Omega  }|u_{i}| ^{p_i(x)}\psi_{j,\varepsilon}dx\to 0
			~~\text{,}~~
		\int_{\Omega  }|u_{i}| ^{q_i(x)}\psi_{j,\varepsilon}dx\to 0	,	
			\
		\int_{\Omega  }|u_{i}| ^{s_i(x)}\psi_{j,\varepsilon}dx\to 0.$$		
		The function $\psi_{j,\varepsilon}$ has compact support, so letting  $m \to +\infty$ and $\varepsilon \to 0$ in relation \eqref{3.2}, we get from relations \eqref{3.3}--\eqref{3.8}, 
		\begin{equation}\label{inq3}
		\begin{alignedat}4 
			0 &=-\lim_{\varepsilon \to 0} \left[
			\limsup_{m\to +\infty } \Big(M_{\theta_i}\left(
			\mathcal{A}_i(u_{i_m})\right)\int_{\Omega  } \bigg(a_{1_i}(|\nabla u_{i_m}| ^{p_i(x)}) |\nabla u_{i_m}| ^{p_i(x)}
			+a_{2_i}(| u_{i_m}| ^{p_i(x)}) | u_{i_m}| ^{p_i(x)} \bigg)
			\psi_{j,\varepsilon}  \,dx\Big)\right]+\overline{\nu}_{i_j}\\
			&\leq -\mathfrak{M}_i^0\lim_{\varepsilon \to 0} \left[
			\limsup_{m\to +\infty } \Big(\int_{\Omega  }\bigg(a_{1_i}(|\nabla u_{i_m}| ^{p_i(x)}) |\nabla u_{i_m}| ^{p_i(x)}
			+a_{2_i}(| u_{i_m}| ^{p_i(x)}) | u_{i_m}| ^{p_i(x)} \bigg)\psi_{j,\varepsilon}  \,dx\Big)\right]+\overline{\nu}_{i_j}\\	
			& \leq -\mathfrak{M}_i^0\lim_{\varepsilon \to 0} \Bigg[
			\limsup_{m\to +\infty } \Bigg(\int_{\Omega  }
			\bigg(\min \{k_{1_i}^0,k_{2_i}^0 \}\big(|\nabla u_{i_m}| ^{p_i(x)} +
			| u_{i_m}| ^{p_i(x)}\big)\\
		&	\quad + \mathcal{K}(k_i^3)\min\{k_{1_i}^2,k_{2_i}^2\}  \big(|\nabla u_{i_m}| ^{q_i(x)} + | u_{i_m}| ^{q_i(x)}\big) \bigg)
			\psi_{j,\varepsilon}  \,dx\Big)\Bigg]+\overline{\nu}_{i_j}.	
				\end{alignedat}
		\end{equation}		
		Note that, when $k_i^3=0$, we have $h_{i}(x)=p_i(x)$. Hence, by using relation \eqref{2.1a1}, we have 	
		\begin{align*}
			0 &\leq  - \mathfrak{M}_i^0\min \{k_{1_i}^0,k_{2_i}^0 \}\lim_{\varepsilon \to 0}
			\int_{\Omega  }\psi_{j,\varepsilon}\,d\mu_i +  \overline{\nu}_{i_j}\\
			&\leq -\mathfrak{M}_i^0\min \{k_{1_i}^0,k_{2_i}^0 \}\mu_{i_j} - \mathfrak{M}_i^0\min \{k_{1_i}^0,k_{2_i}^0 \}\lim_{\varepsilon \to 0}
			\int_{\Omega  }|\nabla u_{i}| ^{p_i(x)}\psi_{j,\varepsilon}\,dx +  \overline{\nu}_{i_j}.
		\end{align*}	
		By applying the Lebesgue Dominated Convergence Theorem, we get
		$\lim_{\varepsilon \to 0}
		\int_{\Omega  }|\nabla u_{i}| ^{p_i(x)}\psi_{j,\varepsilon}\,dx=0.$		
		Therefore, 
		\begin{equation}\label{k0}
			\mathfrak{M}_i^0\min \{k_{1_i}^0,k_{2_i}^0 \}\mu_{i_j}	\leq \overline{\nu}_{i_j}.
		\end{equation}				
		On the other hand, if $k_i^3>0$, then $h_i(x)=q_i(x)$ Therefore, it follows from relations \eqref{GNS} and \eqref{inq3} that
		\begin{align*}
			0 &\leq - \mathfrak{M}_i^0\lim_{\varepsilon \to 0} \left[ \limsup_{m \to 0}
			\Big(\int_{\Omega }\mathcal{K}(k_i^3)\min \{k_{1_i}^2,k_{2_i}^2 \}|\nabla u_{i_m}| ^{q_i(x)}\psi_{j,\varepsilon}\,dx\Big) \right] +  \overline{\nu}_{i_j} \\
			&\leq -\mathfrak{M}_i^0\mathcal{K}(k_i^3)\min \{k_{1_i}^2,k_{2_i}^2 \}\lim_{\varepsilon \to 0}
			\int_{\Omega  }\psi_{j,\varepsilon}\,d\mu_i +  \overline{\nu}_{i_j}\\
			&\leq  -\mathfrak{M}_i^0\mathcal{K}(k_i^3)\min \{k_{1_i}^2,k_{2_i}^2 \}\mu_{i_j} -\mathfrak{M}_i^0\mathcal{K}(k_i^3)\min \{k_{1_i}^2,k_{2_i}^2 \}\lim_{\varepsilon \to 0} \int_{\Omega  }|\nabla u_{i}| ^{p_i(x)}\psi_{j,\varepsilon}\,dx +  \overline{\nu}_{i_j}.
		\end{align*}
		By applying the Lebesgue Dominated Convergence Theorem again, we get 
		$
			\lim_{\epsilon \to 0}
			\int_{\Omega  }|\nabla u_{i}| ^{q_i(x)}\psi_{j,\varepsilon}\,dx=0.
		$
		Hence, 
		\begin{equation}\label{k1}
			\mathfrak{M}_i^0\mathcal{K}(k_i^3)\min \{k_{1_i}^2,k_{2_i}^2 \}\mu_{i_j}	\leq \overline{\nu}_{i_j}.
		\end{equation}
		By combining relations \eqref{k0} and \eqref{k1}, we have
		$\mathfrak{M}_i^0\big((1-\mathcal{K}(k_i^3))\min \{k_{1_i}^0,k_{2_i}^0 \}+ \mathcal{K}(k_i^3)\min \{k_{1_i}^2,k_{2_i}^2 \}\big)\mu_{i_j} \leq \overline{\nu}_{i_j}.$  
		By using relation \eqref{GNS}, we obtain 	
		$$\overline{T_{i}}_{x_j}\overline{\nu}_{i_j}^\frac{1}{h_i^\ast(x_j)}\leq \mu_{i_j}^{\frac{1}{h_i(x_j)}}\leq \Biggl(\frac{\overline{\nu}_{i_j}}{\mathfrak{M}_i^0\big((1-\mathcal{K}(k_i^3))\min \{k_{1_i}^0,k_{2_i}^0 \}+ \mathcal{K}(k_i^3)\min \{k_{1_i}^2,k_{2_i}^2 \}\big)}\Biggr)^{\frac{1}{h_i(x_j)}} .$$
		which implies that $\overline{\nu}_{i_j}=0$ or $\overline{\nu}_{i_j}\geq \overline{T_i}_{x_j}^N\Big(\mathfrak{M}_i^0\big(\min \{k_{1_i}^0,k_{2_i}^0 \}(1-\mathcal{K}(k_i^3)+\mathcal{K}(k_i^3)\min \{k_{1_i}^2,k_{2_i}^2 \})\big)\Big)^{N/h_i(x_j)},$ for all $j \in J_i^1$.
	On the other hand, by using assumptions $(\textbf{\textit{M}})$ and $(\textbf{\textit{F}}_3)$, we have	
		\begin{align*}
			c_{\theta,\lambda }&=E_{\theta, \lambda }(u_{m})-	\langle E_{\theta,\lambda }'(u_{m}),\frac{u_{m}}{\gamma }\rangle\\
			&=\sum_{i=1}^{n}\widehat{M_{\theta_i}}\left(\mathcal{A}_i (u_{i_m})\right)-
			\sum_{i=1}^{n}\displaystyle\int_{\Omega  }\dfrac{1}{s_i(x)}|u_{i_m}|^{s_i(x)}dx
			-\sum_{i=1}^{n}\displaystyle\int_{\partial\Omega  }\dfrac{1}{\ell_i(x)}|u_{i_m}|^{\ell_i(x)}d\sigma_x  - \int_{\Omega}\lambda F(x,u_{m})\,dx
			\\
			&\quad-
			\sum_{i=1}^{n} M_{\theta_i}\left(\mathcal{A}_i (u_{i_m})\right)\int_{\Omega  }
			\frac{1}{{\gamma_i}}\bigg(a_{1_i}(|\nabla u_{i_m}| ^{p_i(x)}) |\nabla u_{i_m}| ^{p_i(x)}
			+ a_{2_i}(| u_{i_m}| ^{p_i(x)}) | u_{i_m}| ^{p_i(x)} \bigg)\,dx \\
			& \quad+ \sum_{i=1}^{n}\int_{\Omega  } \frac{1}{\gamma_i}| u_{i_m}|^{s_i(x)}\,dx
			+ \sum_{i=1}^{n}\int_{\partial \Omega  } \dfrac{1}{\gamma_i}| u_{i_m}|^{\ell_i(x)}\,d\sigma_x
			+ \sum_{i=1}^{n} \int_{\Omega  }\frac{\lambda}{\gamma_i}F_{u_i}(x,u_{m})u_{i_m}\,dx + o_m(1)\\	
			&\geq \sum_{i=1}^{n}\frac{\theta_i \mathfrak{M}_i^0}{p_i^+ \max\{\beta_{1_i},\beta_{2_i}\}} \int_{\Omega  } \bigg(a_{1_i}(|\nabla u_{i_m}| ^{p_i(x)})|\nabla u_{i_m}| ^{p_i(x)}
			+a_{2_i}(| u_{i_m}| ^{p_i(x)})| u_{i_m}| ^{p_i(x)} \bigg)\,dx
			- \sum_{i=1}^{n}\frac{1}{s_i^-} \int_{\Omega  }| u_{i_m}| ^{s_i(x)}\,dx\\
			&\quad	- \sum_{i=1}^{n}\frac{1}{\ell_i^-} \int_{\partial \Omega  }| u_{i_m}| ^{\ell_i(x)}\,d\sigma_x
		- \int_{\Omega}\lambda F(x,u_{m})dx
			- \sum_{i=1}^{n}\frac{\mathfrak{M}_i^0}{\gamma_i} \int_{\Omega  } \bigg(a_{1_i}(|\nabla u_{i_m}| ^{p_i(x)})|\nabla u_{i_m}| ^{p_i(x)}
			+a_{2_i}(| u_{i_m}| ^{p_i(x)})| u_{i_m}| ^{p_i(x)}\bigg)	\,dx
			\\
			&\quad	+\sum_{i=1}^{n}\frac{ 1}{\gamma_i} \int_{\Omega  }|u_{i_m}| ^{s_i(x)}\,dx
			+\sum_{i=1}^{n}\frac{ 1}{\gamma_i} \int_{\partial \Omega  }|u_{i_m}| ^{\ell_i(x)}\,dx
			+ \sum_{i=1}^{n} \int_{\Omega  }\frac{\lambda}{\gamma_i}F_{u_i}(x,u_{m})u_{i_m}\,dx +o_m(1)\\	
			&\geq \sum_{i=1}^{n} \mathfrak{M}_i^0\Big(\frac{\theta_i }{p_i^+\max\{\beta_{1_i},\beta_{2_i}\}}-\frac{ 1}{\gamma_i}\Big)\displaystyle\int_{\Omega  }\bigg(a_{1_i}(|\nabla u_{i_m}| ^{p_i(x)})|\nabla u_{i_m}| ^{p_i(x)}
			+a_{2_i}(| u_{i_m}| ^{p_i(x)})| u_{i_m}| ^{p_i(x)}\bigg)	\,dx\\
			&\quad	+
			\sum_{i=1}^{n} \Big(\frac{ 1}{\gamma_i}-\frac{ 1}{s_i^-}\Big)\displaystyle\int_{\Omega  }|u_{i_m}|^{s_i(x)}dx
			+	\sum_{i=1}^{n} \Big(\frac{ 1}{\gamma_i}-\frac{ 1}{\ell_i^-}\Big)\displaystyle\int_{\partial\Omega  }|u_{i_m}|^{\ell_i(x)}dx
			+ \lambda \displaystyle\int_{\Omega}
			\left[ \sum_{i=1}^{n}  \frac{u_{i_m}}{\gamma_i}F_{ u_i}(x,u_{m})	-F(x,u_{m})
			\right] dx +o_m(1).
		\end{align*}		
		Hence, we have 		
		$ c_{\theta,\lambda} \geq \sum_{i=1}^{n} \Big(\frac{ 1}{\gamma_i}-\frac{ 1}{\ell_i^-}\Big)\displaystyle\int_{\partial \Omega  }|u_{i_m}|^{\ell_i(x)}d\sigma_x + o_m(1).$
		Setting $\mathcal{C}^1_{i\kappa} =\cup_{x\in \mathcal{C}^1_{h_{i}}}(\textbf{B}_{\kappa}(x)\cap \Omega)= \{ x\in \Omega : \text{dist}(x, \mathcal{C}^1_{h_{i}})<\kappa \}$, as $m\to +\infty,$ we obtain	
		\begin{align*}
		c_{\theta,\lambda} &
			\geq \sum_{i=1}^{n} \Big(\frac{ 1}{\gamma_i}-\frac{ 1}{{\ell_i}_{\mathcal{C}^1_{i\kappa}}^-}\Big)\Big(\int_{\Omega  }|u_i|^{\ell_i(x)}dx +\sum_{j\in J_i^1}\nu_{ij}\delta_{x_j}\Big)\\
			& \geq \sum_{i=1}^{n} \Big(\frac{ 1}{\gamma_i}-\frac{ 1}{{\ell_i}_{\mathcal{C}^1_{i\kappa}}^-}\Big)\Big(\int_{\Omega  }|u_i|^{s_i(x)}dx +\inf_{j\in J_i^1}\Big\{\overline{T_{i}}_{x_j}^N\Big(D_i\Big)^{N/h_i(x_j)}\Big\} 
			\operatorname{Card}
			 J_i^1\Big),
		\end{align*}
  where $D_i= \mathfrak{M}_i^0\big(\min\{k_{1_i}^0, k_{2_i}^0\}(1-\mathcal{K}(k_i^3)+\mathcal{K}(k_i^3)\min\{k_{1_i}^2,k_{2_i}^2\})\big)$.	 Since $\kappa > 0$ is arbitrary and $\ell_i$ are continuous functions for all $i\in\{1,2,\dots,n\}$, we get
		\begin{align*}
			c_{\theta, \lambda} &\geq \sum_{i=1}^{n} \Big(\frac{ 1}{\gamma_i}-\frac{ 1}{{\ell_i}_{\mathcal{C}^1_{h_{i}}}^-}\Big)\Big(\int_{\Omega  }|u_i|^{s_i(x)}dx +\inf_{j\in J_i^1}\Big\{\overline{T_{i}}_{x_j}^N\Big(D_i\Big)^{N/h_i(x_j)}\Big\} 
			\operatorname{Card}
			 J_i^1\Big).
		\end{align*}		
		Suppose that $\cup_{i=1}^{n}J_i^1\neq \emptyset$, then
		$$c_{\theta, \lambda}  \geq  \min_{1\leq i\leq n} \left\lbrace \Big(\frac{1}{\gamma_i}-\frac{1}{{\ell_i}_{\mathcal{C}^1_{h_{i}}}^-}\Big)\inf_{j\in J_i^1}\Big\{\overline{T_{i}}_{x_j}^N\Big(D_i\Big)^{N/h_i(x_j)}\Big\} \right\rbrace.$$		
		 Therefore, if $\displaystyle c_{\theta, \lambda} < \min_{1\leq i\leq n} \left\lbrace \Big(\frac{1}{\gamma_i}-\frac{1}{{\ell_i}_{\mathcal{C}^1_{h_{i}}}^-}\Big)\inf_{j\in J_i^1}\Big\{\overline{T_{i}}_{x_j}^N\Big(D_i\Big)^{N/h_i(x_j)}\Big\} \right\rbrace$, the set $\cup_{i=1}^{n}J_i^1$ is empty, which means that for all $1\leq i\leq n,$ 
		$ \|u_{i_m}\|_{L^{\ell_i(x)}(\partial\Omega)} \to \|u_i\|_{L^{\ell_i(x)}(\partial \Omega)}$. Since  $u_m\rightharpoonup u$
		in $X$, we have for all $i \in \left\lbrace 1,2,\dots,n\right\rbrace $ that $u_{i_m} \to u_i$ strongly in $L^{\ell_i(x)}(\partial \Omega)$ .		
		Next, consider $J_i^2\neq \emptyset$, by the same approach for the case $J_i^1$. We have		
		$$c_{\theta, \lambda}  \geq  \min_{1\leq i\leq n} \left\lbrace \Big(\frac{1}{\gamma_i}-\frac{1}{{s_i}_{\mathcal{C}^2_{h_{i}}}^-}\Big)\inf_{j\in J_i^2}\Big\{S_{i}^N\Big(D_i\Big)^{N/h_i(x_j)}\Big\} \right\rbrace,$$
	where $D_i= \mathfrak{M}_i^0\big(\min\{k_{1_i}^0, k_{2_i}^0\}(1-\mathcal{K}(k_i^3)+\mathcal{K}(k_i^3)\min\{k_{1_i}^2,k_{2_i}^2\})\big)$. Hence, we deduce that $\cup_{i=1}^nJ_i^2= \emptyset$, which means that for all $1\leq i\leq n,$
		$ \|u_{i_m}\|_{L^{s_i(x)}(\Omega)} \to \|u_i\|_{L^{s_i(x)}(\Omega)}$ . Since $u_m\rightharpoonup u$
		in $X$, we have $u_{i_m} \to u_i$ strongly in $L^{s_i(x)}(\Omega),$ for all $i \in \left\lbrace 1,2,\dots,n\right\rbrace $.
		On the other hand, we have 
			\begin{align*}
			 E'_{\theta,\lambda}& (u_{1_m},\dots,u_{n_m})-\left\langle E'_{\theta,\lambda} (u_{1_k},\dots,u_{n_k}),(u_{1_m}-u_{1_k},0,\dots,0)	\right\rangle \\
			&=\left\langle \Phi_{\theta}' (u_{1_m},\dots,u_{n_m})-\Phi_{\theta}' (u_{1_k},\dots,u_{n_k}),(u_{1_m}-u_{1_k},0,\dots,0)	\right\rangle\\
		&\quad	-\left\langle \Psi  ' (u_{1_m},\dots,u_{n_m})-\Psi' (u_{1_k},\dots,u_{n_k}),(u_{1_m}-u_{1_k},0,\dots,0)\right\rangle\\
			&\quad-\left\langle	\Upsilon' (u_{1_m},\dots,u_{n_m})-\Upsilon' (u_{1_k},\dots,u_{n_k}),(u_{1_m}-u_{1_k},0,\dots,0)\right\rangle\\
		&\quad	-\left\langle \mathcal{F}_{\lambda} ' (u_{1_m},\dots,u_{n_m})-\mathcal{F}_{\lambda}' (u_{1_k},\dots,u_{n_k}),(u_{1_m}-u_{1_k},0,\dots,0)	\right\rangle,
		\end{align*}
		thus $E_{\theta,\lambda}'(u_{1_m},\dots,u_{n_m})\to 0$, i.e $E_{\theta,\lambda} '(u_{1_m},\dots,u_{n_m})$  is a Cauchy sequence in $X^\ast$.
		Furthermore,  by using H\"{o}lder's inequality again, we find		
			\begin{align*}
			\left\langle \Psi ' (u_{1_m},\dots,u_{n_m})-\Psi' (u_{1_k},\dots,u_{n_k}),(u_{1_m}-u_{1_k},0,\dots,0)	\right\rangle &
			=\int_{\Omega }\Big(|u_{1_m}|^{s_1(x)-2}u_{1_m}- |u_{1_k}|^{s_1(x)-2}u_{1_k}\Big)(u_{1_m}-u_{1_k})dx\\
		&	\leq \bigg\| |u_{1_m}|^{s_1(x)-2}u_{1_m}- |u_{1_k}|^{s_1(x)-2}u_{1_k} \bigg\|_{L^{s_1'(x)}(\Omega)} \bigg\| u_{1_m}-u_{1_k}\bigg\|_{L^{s_1(x)}(\Omega)}.
		\end{align*}
		Similarly, we also have 
		\begin{align*}
			\left\langle \Upsilon ' (u_{1_m},\dots,u_{n_m})-\Upsilon' (u_{1_k},\dots,u_{n_k}),(u_{1_m}-u_{1_k},0,\dots,0)	\right\rangle
			&=\int_{\partial\Omega }\Big(|u_{1_m}|^{\ell_1(x)-2}u_{1_m}- |u_{1_k}|^{\ell_1(x)-2}u_{1_k}\Big)(u_{1_m}-u_{1_k})d\sigma_x\\
			&\leq \bigg\|  |u_{1_m}|^{\ell_1(x)-2}u_{1_m}- |u_{1_k}|^{\ell_1(x)-2}u_{1_k} \bigg\|_{L^{\ell_1'(x)}(\partial\Omega)} \bigg\| u_{1_m}-u_{1_k}\bigg\|_{L^{\ell_1(x)}(\partial \Omega)}.
		\end{align*}		
		Since  $\left\lbrace u_{1_m}\right\rbrace $ is  a Cauchy sequence in $L^{s_1(x)}(\Omega)$ and  $L^{\ell_1(x)}(\partial \Omega)$,  it follows that $ \Psi'(u_{1_m},\dots,u_{n_m})$ and $\Upsilon '(u_{1_m},\dots,u_{n_m})$ are Cauchy sequences in $X^\star$.	
	Moreover, by  compactness of $\mathcal{F}_{\lambda} '$, we have 
		$ (u_{1_m},\dots,u_{n_m})\rightharpoonup (u_1,\dots,u_n)  \Rightarrow
		\mathcal{F}_{\lambda}  '(u_{1_m},\dots,u_{n_m}) \rightarrow \mathcal{F}_{\lambda} ' (u_1,\dots,u_n),$
		which means that $\mathcal{F}_{\lambda} '(u_{1_m},\dots,u_{n_m})$ is a Cauchy sequence also in $X^\ast$.	
	Therefore, invoking some elementary inequalities (see, e.g., Hurtado et al. \cite[Auxiliary Results]{Hurtado}), we conclude that for all $ \varrho , \zeta \in \mathbb{R}^N, $
		\begin{eqnarray}
			\label{ineq1}
			\begin{cases}
				|\varrho  -\zeta|^{p_i(x)}\leq c_{p_i} \Big(\mathcal{B}_{j_i}(\varrho)  -\mathcal{B}_{j_i}(\zeta) \Big) \cdot\left(\varrho -\zeta\right)  &\quad\text{if } p_i(x) \geq2\\
				
				|\varrho -\zeta|^{2}\leq c (|\varrho| +|\zeta|)^{2-p_i(x)}\Big(\mathcal{B}_{j_i}(\varrho)  -\mathcal{B}_{j_i}(\zeta) \Big)\cdot\left(\varrho -\zeta\right)  	&\quad\text{if } 1<p_i(x) <2 \\
			\end{cases}
		\end{eqnarray}
		where $\cdot$ denotes the standard inner product  in $\mathbb{R}^N$. Define the subsets of $\Omega$ dependent on $p_i$ by
		$
		U_{p_i}:= \big\{ x\in \Omega : p(x)\geq 2\big\}
		$
		 and
		 $
		V_{p_i}:= \big\{ x\in \Omega : 1<p(x)< 2\big\}.
		$
		 For $i=1,$ replacing $\varrho$ and $\zeta$ by $\nabla u_{1_m}$ and $\nabla u_{1_k},$ respectively when $j=1$ and by $ u_{1_m}$ and $ u_{1_k},$ respectively, when $j=2$ in the first line of relation \eqref{ineq1}, and integrating over $\Omega$, we obtain		
	\[	
	 c_1\int_{U_{p_i} }\big(|\nabla u_{1_m} - \nabla u_{1_k}|^{p_1(x)}+|u_{1_m} -u_{1_k}|^{p_1(x)}\big)dx \leq \left\langle  \Phi_{\theta}'(u_{1_m},\dots,u_{n_m})-\Phi_{\theta}'(u_{1_k},\dots,u_{n_k})
		,(u_{1_m}-u_{1_k},0,\dots,0)\right\rangle.
		\]		
	On the other hand, by the second line of relation \eqref{ineq1}, we have 
		\begin{multline*}
			c_2\int_{V_{p_i} }\bigg(\sigma_1(x)^{p_1(x)-2}| \nabla u_{1_m} - \nabla u_{1_k}|^{2} +
			\sigma_2(x)^{p_1(x)-2}|u_{1_m} -u_{1_k}|^{2}
			\bigg)dx\\
			\leq  \left\langle \Phi_{\theta}'(u_{1_m},\dots,u_{n_m}) -\Phi_{\theta}'(u_{1_k},\dots,u_{n_k}), (u_{1_m} -u_{1_k},0,\dots,0)\right\rangle,
		\end{multline*}	
	where $\sigma_1(x) =(|\nabla u_{1_m}|+|\nabla u_{1_k}|)$ and $\sigma_2(x) =(|u_{1_m}|+| u_{1_k}|)$. Hence, by H\"{o}lder's inequality and Lemma \ref{prop22}, 		
		\begin{align*}
		\int_{V_{p_i} }\bigg(&|\nabla u_{1_m} - \nabla u_{1_k}|^{p_1(x)}+|u_{1_m} -u_{1_k}|^{p_1(x)}\bigg)dx\\
			&  = \int_{V_{p_i}}\sigma_1^{\frac{p_1(x)(p_1(x)-2)}{2}}\Big(\sigma_1^{\frac{p_1(x)(p_1(x)-2)}{2}}|\nabla u_{1_m} -\nabla u_{1_k}|^{p_1(x)} \Big)dx		+	\int_{V_{p_i}}\sigma_1^{\frac{p_1(x)(p_1(x)-2)}{2}}\Big(\sigma_1^{\frac{p_1(x)(p_1(x)-2)}{2}}|u_{1_m} -u_{1_k}|^{p_1(x)} \Big)dx	\\
		&	\leq C_3 \bigg\| \sigma_1^{\frac{p_1(x)(2-p_1(x))}{2}} \bigg\|_{L^{\frac{2}{2-p_1(x)}}(V_{p_i})}
			\bigg\| \sigma_1^{\frac{p_1(x)(p_1(x)-2)}{2}}|\nabla u_{1_m} - \nabla u_{1_k}|^{p_1(x)} \bigg\|_{L^{\frac{2}{p_1(x)}}(V_{p_i})}\\
		& \qquad	+ C_4 \bigg\| \sigma_2^{\frac{p_1(x)(2-p_1(x))}{2}} \bigg\|_{L^{\frac{2}{2-p_1(x)}}(V_{p_i})}
			\bigg\|\sigma_2^{\frac{p_1(x)(p_1(x)-2)}{2}}|u_{1_m} -u_{1_k}|^{p_1(x)} \bigg\|_{L^{\frac{2}{p_1(x)}}(V_{p_i})}\\
			&\leq C_5\max\Biggl\{ \bigg\| \sigma_1\bigg\|^{\big(\frac{p_1(x)(p_1(x)-2)}{2}\big)^-}_{L^{p_1(x)}(V_{p_i})}, \bigg\| \sigma_1\bigg\|^{\big(\frac{p_1(x)(p_1(x)-2)}{2}\big)^+}_{L^{p_1(x)}(V_{p_i})}
			\Biggr\} \\
			& \quad \quad {} \times \max\Biggl\{ \Big(	\int_{V_{p_i}}\sigma_1^{p_1(x)-2}|\nabla u_{1_m} - \nabla u_{1_k}|^{2}dx\Big)^{\frac{p_1^-}{2}}, \Big(	\int_{V_{p_i}}\sigma_1^{p_1(x)-2}| \nabla u_{1_m} -\nabla u_{1_k}|^{2}dx\Big)^{\frac{p_1^+}{2}}
			\Biggr\}\\
		& \qquad	+ C_6\max\Biggl\{ \Vert \sigma_2\Vert^{[\frac{p_1(x)(p_1(x)-2)}{2}]^-}_{L^{p_1(x)}(V_{p_i})}, \Vert \sigma_2\Vert^{[\frac{p_1(x)(p_1(x)-2)}{2}]^+}_{L^{p_1(x)}(V_{p_i})}
			\Biggr\}\\  
			&  \quad \quad {} \times  \max\Biggl\{ \Big(	\int_{V_{p_i}}\sigma_1^{p_1(x)-2}|u_{1_m} -u_{1_k}|^{2}dx\Big)^{\frac{p_1^-}{2}}, \Big(	\int_{V_{p_i}}\sigma_1^{p_1(x)-2}|u_{1_m} -u_{1_k}|^{2}dx\Big)^{\frac{p_1^+}{2}}
			\Biggr\}.
		\end{align*}			
		Since $\left\lbrace u_{1_m}\right\rbrace $  is bounded sequence  in
		$W^{1,h_1(x)}(\Omega)\cap W^{1,p_1(x)}, $ we have 
		$$\left\langle \Phi_{\theta}'(u_{1_m},\dots,u_{n_m}) -\Phi_{\theta}'(u_{1_k},\dots,u_{n_k}), (u_{1_m} -u_{1_k},0,\dots,0)\right\rangle \to 0, \ \text{as } ~m,k \to  +\infty,$$
	hence
	$\left\lbrace u_{1_m}\right\rbrace $  is a Cauchy sequence in $W^{1,p_1(x)}\cap W^{1,h_1(x)}(\Omega)$. 
	We argue similarly for $\left\lbrace u_{i_m}\right\rbrace, $  
		$$\left\langle \Phi_{\theta}'(u_{1_m},\dots,u_{i_m},\dots,u_{n_m}) -\Phi_{\theta}'(u_{1_k},\dots,u_{i_k},\dots,u_{n_k}), (0,\dots,u_{i_m}-u_{i_k},0,\dots,0)\right\rangle,
		\		
		\hbox{for all}
		\
		i\in \{2,\dots,n\}.$$  Thus, we can conclude that $u_m= (u_{1_m},\dots,u_{n_m})\to u= (u_{1},\dots,u_{n})$ strongly in $X$ as $m\to +\infty$. Therefore, we have that
		$E_{\theta, \lambda}(u)= c_{\theta, \lambda}>0$ and 	$E_{\theta, \lambda}'(u)=0$	 in $X'$, i.e., $u\in X$ is weak solution of  problem \eqref{s.aux1.1}.  Since $E_{\theta,\,\lambda}(u)=c_{\theta,\,\lambda}>0=E_{\theta,\,\lambda}(0),$ we can conclude that $u\not\equiv 0$.
	\end{proof}	
\section{\sc Proof of the main theorem}\label{proof res}
	Now we are in position to prove Theorem \ref{thm.result}.	
	\begin{proof}  	
	Invoking Theorem \ref{AuxTh}, for all $\lambda\geq\lambda_*$ let $u_\lambda=(u_{1,\lambda},u_{2,\lambda},\dots,u_{n,\lambda})$ be a solution of system \eqref{s.aux1.1}. We shall prove that
	\begin{equation}\label{claim2}
		\mbox{there exists }\,\, \lambda^*\geq\lambda_*\,\, \mbox{such that}\,\, \mathcal{A}_i(u_{i,\lambda})\leq \tau_i^0, \,\,\mbox{for all}\,\,\lambda\geq\lambda^* ,
	\end{equation}
	where $\tau_i^0$ is defined as at the beginning of Section \ref{sec auxiliary}.	 We argue by contradiction and suppose that there is a sequence $\left\{\lambda_m\right\}_{m\in\mathbb{N}}\subset\mathbb{R}$ such that $\mathcal{A}_i(u_{i,\lambda_m})\geq \tau_i^0,$ for all $i \in \{1,2,\dots,n\}$. 		
By assumption $(\textbf{\textit{A}}_1)$ and the fact  $\mathcal{A}_i(u_{i,\lambda_m})\geq \tau_i^0$, we get
			\begin{equation} \label{keyA_1}
	\displaystyle \int_{\Omega }\Big(\max\{k_{1_i}^1,k_{2_i}^1\} \big(|\nabla u_{i,\lambda_m}|^{p_i(x)} + | u_{i,\lambda_m}|^{p_i(x)} \big)
	+
	k_i^3\big(|\nabla u_{i,\lambda_m}|^{q_i(x)}
	+ |u_{i,\lambda_m}|^{q_i(x)}\big)\Big)dx 
	\geq \tau_i^0 \text{ for all } i=1,2,\dots,n.
	\end{equation}		
 Since $u_{\lambda_m}=(u_{1,\lambda_m},\dots,u_{n,\lambda_m})$ is a critical point of the functional $E_{\theta,\,\lambda_m}$, we can conclude, using assumptions $(\textbf{\textit{M}})$ and $(\textbf{\textit{F}}_3)$, that				
		\begin{equation}\label{Inq43}
			\begin{alignedat}2
				c_{\theta,\lambda_m }&=E_{\theta,\lambda }(u_{\lambda_m})-	\big\langle E_{\theta,\lambda }'(u_{\lambda_m}),\frac{u_{\lambda_m}}{\gamma }\big\rangle\\
			&\geq\sum_{i=1}^{n}\widehat{M_{\theta_i}}\left(\mathcal{A}_i (u_{i,\lambda_m})\right)-
			\sum_{i=1}^{n} M_i\left(\mathcal{A}_i (u_{i,\lambda_m})\right)\int_{\Omega  }\frac{1}{\gamma_i}\bigg(a_{1_i}(|\nabla u_{i,\lambda_m}| ^{p_i(x)})|\nabla u_{i,\lambda_m}| ^{p_i(x)} + a_{2_i}(| u_{i,\lambda_m}| ^{p_i(x)})|u_{i,\lambda_m}|^{p_i(x)} \bigg)  \,dx \\
			&\geq \sum_{i=1}^{n}\frac{\mathfrak{M}_i^0}{p_i^+\max\{\beta_{1_i}, \beta_{2_i}\}} \int_{\Omega  }\bigg(a_{1_i}(|\nabla u_{i,\lambda_m}| ^{p_i(x)})|\nabla u_{i,\lambda_m}| ^{p_i(x)} + a_{2_i}(| u_{i,\lambda_m}| ^{p_i(x)})|u_{i,\lambda_m}| ^{p_i(x)} \bigg)  \,dx\\
			&\quad- \sum_{i=1}^{n}\frac{\theta_i}{\gamma_i} \int_{\Omega  }\bigg(a_{1_i}(|\nabla u_{i,\lambda_m}| ^{p_i(x)})|\nabla u_{i,\lambda_m}| ^{p_i(x)} + a_{2_i}(| u_{i,\lambda_m}| ^{p_i(x)})|u_{i,\lambda_m}| ^{p_i(x)} \bigg)  \,dx\\	
			&\geq \sum_{i=1}^{n} \Big(\frac{ \mathfrak{M}_i^0}{p_i^+\max\{\beta_{1_i}, \beta_{2_i}\}}-\frac{ \theta_i}{\gamma_i}\Big)\displaystyle\int_{\Omega  }\bigg(a_{1_i}(|\nabla u_{i,\lambda_m}| ^{p_i(x)})|\nabla u_{i,\lambda_m}| ^{p_i(x)} + a_{2_i}(| u_{i,\lambda_m}| ^{p_i(x)})|u_{i,\lambda_m}| ^{p_i(x)} \bigg) \,dx\\
			&\geq \sum_{i=1}^{n} \Big(\frac{ \mathfrak{M}_i^0}{p_i^+\max\{\beta_{1_i}, \beta_{2_i}\}}-\frac{ \theta_i}{\gamma_i}\Big) \Bigg[\displaystyle \int_{\Omega}\max\{k_{1_i}^0,k_{2_i}^0\} \big(|\nabla u_{i,\lambda_m}| ^{p_i(x)} + |u_{i,\lambda_m}| ^{p_i(x)}\big)\,dx \\
		&\quad	+ \mathcal{K}(k_i^3)\max\{k_{1_i}^2,k_{2_i}^2\}
			\int_{\Omega}\big(| \nabla u_{i,\lambda_m} | ^{q_i(x)}+ | u_{i,\lambda_m}| ^{q_i(x)}\big)\,dx\Bigg].
			\end{alignedat}
		\end{equation}
If $k_i^3=0$, then using relation \eqref{keyA_1}, we find	
	$	\displaystyle \int_{\Omega} \big(|\nabla u_i| ^{p_i(x)} + |u_i| ^{p_i(x)}\big)\,dx	\geq \frac{\tau_i^0}{\max\{k_{1_i}^1,k_{2_i}^1\}},$	
	thus we have 
		\begin{align*}
		c_{\theta,\lambda }&\geq \sum_{i=1}^{n} \Big(\frac{ \mathfrak{M}_i^0}{p_i^+\max\{\beta_{1_i}, \beta_{2_i}\}}-\frac{ \theta_i}{\gamma_i}\Big) \Big(\frac{\max\{k_{1_i}^0,k_{2_i}^0\}}{\max\{k_{1_i}^1,k_{2_i}^1\}}\Big)\tau_i^0 >0.
\end{align*}			
This contradicts Lemma \ref{lemma34}, because $\displaystyle\lim_{m\to +\infty }c_{\theta,\lambda_m}=0$.
On the other hand, if $k_i^3>0$, we multiplying relation \eqref{Inq43} by  $\displaystyle\max_{1\leq i\leq n}\{\max\{k_{1_i}^1,k_{2_i}^1\}\times k_i^3\}>0,$ and by using also relation \eqref{keyA_1}, we  get	
		\begin{align*}
	\displaystyle\max_{1\leq i\leq n}\{\max\{k_{1_i}^1,k_{2_i}^1\}\times k_i^3\}	C_{\theta,\lambda }&\geq \sum_{i=1}^{n} \Big(\frac{ \mathfrak{M}_i^0}{p_i^+\max\{\beta_{1_i}, \beta_{2_i}\}}-\frac{ \theta_i}{\gamma_i}\Big)\kappa_i \displaystyle \int_{\Omega }\Bigg(\max\{k_{1_i}^1,k_{2_i}^1\} \big(|\nabla u_{i,\lambda_m}|^{p_i(x)} + | u_{i,\lambda_m}|^{p_i(x)} \big)\\
	&\qquad+
	k_i^3\big(|\nabla u_{i,\lambda_m}|^{q_i(x)}
	+ |u_{i,\lambda_m}|^{q_i(x)}\big)\Bigg)dx\\
	&\geq  \sum_{i=1}^{n} \Big(\frac{ \mathfrak{M}_i^0}{p_i^+\max\{\beta_{1_i}, \beta_{2_i}\}}-\frac{ \theta_i}{\gamma_i}\Big)\kappa_i \tau_i^0,
	\end{align*}	
where $\kappa_i=\min \bigg\{\max\{k_{1_i}^0,k_{2_i}^0\}\times k_i^3,\max\{k_{1_i}^1,k_{2_i}^1\}\times\max\{k_{1_i}^2,k_{2_i}^2\}\bigg\}$. This also contradicts Lemma \ref{lemma34} because $\displaystyle\lim_{m\to +\infty }c_{\theta,\lambda}=0$. Hence, we can conclude in both cases, that there exists $\lambda^{\ast}\geq \lambda_\ast$ such that $ \mathcal{A}_i(u_{i,\lambda})\geq \tau_i^0,$ for all $\lambda \geq \lambda^{\ast}$. So, we
can
 find $M_{\theta_i}(\mathcal{A}_i(u_\lambda))=M_{\theta_i}(\mathcal{A}_i(u_\lambda)),$ for all $\lambda \geq \lambda^{\ast}$,  which implies that 
	$E_{\theta,\lambda}(u_\lambda)=E_{\lambda}(u_\lambda)$ and 
	$E_{\theta,\lambda}'(u_\lambda)=E_{\lambda}'(u_\lambda)$, that is $u_\lambda$ is a nontrivial weak solution of the problem \eqref{s1.1}, for each $\lambda\geq \lambda^{\ast}$.	\\	
	It now remains to consider the asymptotic behavior of solutions to problem
	\eqref{s1.1}. By assumptions $(\textbf{\textit{A}}_2)$, $(\textbf{\textit{A}}_4)$, $(\textbf{\textit{M}})$, $(\textbf{\textit{F}}_3)$, and inequalities \eqref{L22}-\eqref{L23} and \eqref{hypM}, arguing as above, we obtain
	\begin{align*}
c_\lambda &=	c_{\theta,\lambda }\\
&\geq \sum_{i=1}^{n} \Big(\frac{ \mathfrak{M}_i^0}{p_i^+\max\{\beta_{1_i}, \beta_{2_i}\}}-\frac{ \theta_i}{\gamma_i}\Big) \Bigg(\displaystyle \min \{k_{1_i}^0,k_{2_i}^0\}\int_{\Omega} \big(| \nabla u_{i,\lambda}| ^{p_i(x)} + | u_{i,\lambda}| ^{p_i(x)}\big)dx \\
& \quad + \mathcal{K}(k_i^3)\min \{k_{1_i}^2,k_{2_i}^2\}
\int_{\Omega} \big(| \nabla u_{i,\lambda}| ^{q_i(x)} + | u_{i,\lambda}| ^{q_i(x)}\big)dx\Bigg)\\
&\geq \sum_{i=1}^{n} \Big(\frac{ \mathfrak{M}_i^0}{p_i^+\max\{\beta_{1_i}, \beta_{2_i}\}}-\frac{ \theta_i}{\gamma_i}\Big)\Bigg[
\min \{k_{1_i}^0,k_{2_i}^0\} \min \bigg\{\|  u_{i,\lambda}\|_{1,p_i(x)}^{p_i^-},\|  u_{i,\lambda}\|_{p_i(x)}^{p_i^+}
\bigg\}    \\
&\quad + \mathcal{K}(k_i^3)\min \{k_{1_i}^2,k_{2_i}^2\} \min \bigg\{\|  u_{i,\lambda}\|_{1,q_i(x)}^{q_i^-},\|  u_{i,\lambda}\|_{1,q_i(x)}^{q_i^+}
\bigg\} \Bigg].
\end{align*}				
		Hence, by Lemma \ref{lemma34}, we get	
		$\lim_{\lambda \to +\infty } \|u_{\lambda}\|= \lim_{ \lambda \to +\infty }\max_{1\leq i\leq n} \bigg\{\| u_i\|_{1,p_i(x)}+ \mathcal{K}(k_i^3)\|  u_i\|_{1,q_i(x)}\bigg\}= 0.$
	\end{proof}
\section{Some examples} \label{Examples}
In the last section, we shall exhibit some examples which are interesting from the mathematical point of view and have
a wide range of applications in physics and 
other scientific fields that fall within the general class of systems studied in this paper, under adequate assumptions on functions $a_{i_j}$.
\begin{example} \label{example5.1}
Taking $a_{1_i}\equiv 1$ and $a_{2_i}\equiv 1$, we see that $a_{1_i}$ satisfies the  assumptions $(\textbf{\textit{A}}_1),(\textbf{\textit{A}}_2),$ and $(\textbf{\textit{A}}_3),$ with $k^0_{j_i}=k^1_{j_i}=1,$
 $k^2_{j_i}>0,$ 
 and 
 $k^3_{i}=0,$ for all $i \in \{1,2,\dots,n\}$ and $j=1$ or $2$. 
 Hence,  system \eqref{s1.1} becomes
	\begin{equation}
	\label{exa5.1}
	\begin{gathered}
		-M_i\Big(\mathcal{A}_i(u_i)\Big)\Big(\Delta_{p_i(x)}u_i -| u_i|^{p_i(x)-2} u_i\Big)
		= |u_i|^{s_i (x)-2}u_i +\lambda F_{u_i}(x,u)    \quad \text{ in } \Omega, \\
		  	M_i\Big(\mathcal{A}_i(u_i)\Big)|\nabla u_i|^{p(x)-2}\nabla u_i.\mathfrak{N}_i=|u_i|^{\ell_i (x)-2}u_i  \qquad\text{ on } \partial \Omega,
	\end{gathered}
\end{equation}
for $1\leq i\leq n (n\in \mathbb{N}^\ast)$, where 
$$\mathcal{A}_i(u_i)=\displaystyle\int_{\Omega }\dfrac{1}{p_i(x)}\big(|\nabla u_i|^{p_i(x)}+ |u_i|^{p_i(x)} \big)dx.$$
 The operator $ \Delta_{p_i(x)}u_i:=\textrm{div}\,(|\nabla u_i|^{p_i(x)-2}\nabla u_i)$ is so-called $p_i(x)$-Laplacian, which coincides with the usual $p_i$-Laplacian when $p_i(x)=p_i$, and with the Laplacian when $p_i(x)=2$.
\end{example} 
\begin{example} \label{example5.2}
Taking $a_{j_i}(\xi)= 1+ \xi^{\frac{q_i(x)-p_i(x)}{p_i(x)}}$,
we see that $a_{j_i}$ satisfies the assumptions $(\textbf{\textit{A}}_1),(\textbf{\textit{A}}_2),$ and $(\textbf{\textit{A}}_3),$ with $k^0_{j_i}=k^1_{j_i}=k^2_{j_i}=k^3_{i}=1,$
for all $i\in \{1,2,\dots,n\}$ and $j=1$ or $2$. Hence, system \eqref{s1.1} becomes  the
following $p\&q$-Laplacian system
\begin{equation}
	\label{exa5.2}
	\begin{gathered}
		-M_i\Big(\mathcal{A}_i(u_i)\Big)\Big(\Delta_{p_i(x)}u_i +\Delta_{q_i(x)}u_i - (| u_i|^{p_i(x)-2} u_i +| u_i|^{q_i(x)-2} u_i) \Big)
		= |u_i|^{s_i (x)-2}u_i +\lambda F_{u_i}(x,u)    \quad \text{ in } \Omega, \\
		M_i\Big(\mathcal{A}_i(u_i)\Big)\Big(|\nabla u_i|^{p_i(x)-2}\nabla u_i + |\nabla u_i|^{q_i(x)-2}\nabla u_i\Big).\mathfrak{N}_i=|u_i|^{\ell_i (x)-2}u_i  \qquad\text{ on } \partial \Omega,
	\end{gathered}
\end{equation}
for $1\leq i\leq n (n\in \mathbb{N}^\ast)$, where $$\mathcal{A}_i(u_i)=\displaystyle \int_{\Omega }\Big(\dfrac{1}{p_i(x)}\left(|\nabla u_i|^{p_i(x)}+ w_i(x)|u_i|^{p_i(x)} \right)+\dfrac{1}{q_i(x)}\left(|\nabla u_i|^{q_i(x)}+ |u_i|^{q_i(x)} \right)\Big)dx.$$
As explained in Cherfils and Il'yasov  \cite{Cherfils}, the study of system \eqref{exa5.2} was motivated by the following 
more general reaction-diffusion system
\begin{equation*} \label{DCE}
	u_t= \textrm{div}[H(u)\nabla u ]+ d(x,u),
	\
	\hbox{where }
	\
	H(u)=|\nabla u|^{p(x)-2}+|\nabla u|^{q(x)-2},
\end{equation*}
which has applications in  biophysics (see, e.g., Fife  \cite{Fife}, Murray  \cite{Myers}), plasma physics (see, e.g., Wilhelmsson  \cite{Wilhelmsson}), and chemical reactions design (see, e.g., Aris \cite{Aris}). In these applications, $u$ represents a concentration, $\textrm{div}[H(u)\nabla u ]$ is the diffusion with diffusion coefficient, and the reaction term $d(x,u)$ relates to source and loss processes. For further details we refer the interested reader to e.g., 
Mahshid and Razani  \cite{Mahshid},  He and Li  \cite{He}, and the references therein.\\
We continue with other examples which are also interesting from the mathematical point of view.
\end{example}
\begin{example}\label{example5.3} 
Taking $a_{j_i}(\xi)= 1+ \frac{\xi}{\sqrt{1+\xi^2}}$ and $a_{2_i}\equiv 1$,
we see that $a_{j_i}$ satisfies the assumptions $(\textbf{\textit{A}}_1),(\textbf{\textit{A}}_2),$ and $(\textbf{\textit{A}}_3),$ with 
$k^0_{1_i}=k^0_{2_i}=k^1_{2_i}=1$,
$k^1_{1_i}=2$, and $k^3_{i}=0$, 
$k^2_{1_i}>0,$
 and 
 $k^2_{2_i}>0,$ 
 for all 
 $i\in \{1,2,\dots,n\}$.
 Hence, system \eqref{s1.1} becomes
\begin{equation}
	\label{exa5.3}
	\begin{gathered}
		-M_i\Big(\mathcal{A}_i(u_i)\Big)\Bigg(\textrm{div}\,\Big(\Big(1+ \frac{|\nabla u_i|^{p_i(x)}}{\sqrt{1+|\nabla u_i|^{2p_i(x)}}}\Big)|\nabla u_i|^{p_i(x)-2}\nabla u_i\Big) -w_i(x)| u_i|^{p_i(x)-2} u_i \Bigg)
		= |u_i|^{s_i (x)-2}u_i +\lambda F_{u_i}(x,u)    \text{ in } \Omega, \\
		M_i\Big(\mathcal{A}_i(u_i)\Big)\Big(1+ \frac{|\nabla u_i|^{p_i(x)}}{\sqrt{1+|\nabla u_i|^{2p_i(x)}}}\Big)|\nabla u_i|^{p_i(x)-2}\nabla u_i.\mathfrak{N}_i=|u_i|^{\ell_i (x)-2}u_i  \qquad\text{ on } \partial \Omega,
	\end{gathered}
\end{equation}
for $1\leq i\leq n (n\in \mathbb{N}^\ast)$, where
\[
\mathcal{A}_i(u_i)=\displaystyle\int_{\mathbb{R}^N }\frac{1}{p_i(x)}\Big(|\nabla u_i|^{p_i(x)} + \sqrt{1+|\nabla u_i|^{2p_i(x)}}+ |u_i|^{p_i(x)}\Big)dx.
\]
	The operator  $\textrm{div}\,\Big(\Big(1+ \frac{|\nabla u|^{p(x)}}{\sqrt{1+|\nabla u|^{2p(x)}}}\Big)|\nabla u|^{p(x)-2}\nabla u\Big)$ is said to be
	   $p_i(x)$-Laplacian like or is called a generalized capillary operator. The capillarity can be briefly explained by considering the effects of two opposing forces: adhesion, i.e. the attractive (or repulsive) force between the molecules of the liquid and those of the container; and cohesion, i.e. the attractive force between the molecules of the liquid. The study of capillary phenomenon has gained much attention. This increasing interest is motivated not only by the fascination in naturally	occurring phenomena, such as motion of drops, bubbles and waves, but also by its importance in applied fields, raging from industrial and biomedical and pharmaceutical to microfluidic systems - for further details we refer the interested reader to e.g.,	Ni and Serrin  \cite{Ni}, and the references therein.
\end{example}
\begin{example}\label{xample5.4}
 Taking  $\displaystyle a_{1_i}(\xi)= 1+ \frac{1}{(1+\xi)^{\frac{p_i(x)-2}{p_i(x)}}}$ and $a_{2_i}\equiv 1$, we see that $a_{j_i}$ satisfies the assumptions $(\textbf{\textit{A}}_1),(\textbf{\textit{A}}_2),$ and $(\textbf{\textit{A}}_3),$ with 
 $k^0_{1_i}=k^0_{2_i}=k^1_{2_i}=1$,
 $k^1_{1_i}=2$, 
 $k^3_{i}=0$, 
 $k^2_{1_i}>0,$ 
 and 
 $k^2_{2_i}>0,$ 
 for all $i\in \{1,2,...,n\}$.  
 Hence, system \eqref{s1.1}  becomes
\begin{equation}
	\label{exa5.4}
	\begin{gathered}
		-M_i\Big(\mathcal{A}_i(u_i)\Big)\Bigg(\textrm{div}\,\Bigg(|\nabla u_i|^{p(x)-2}\nabla u_i + \dfrac{|\nabla u_i|^{p(x)-2}\nabla u_i}{(1+ |\nabla u_i|^{p(x)})^{\frac{p(x)-2}{p(x)}}}\Bigg) -| u_i|^{p_i(x)-2} u_i \Bigg)
		= |u_i|^{s_i (x)-2}u_i +\lambda F_{u_i}(x,u)    \quad \text{ in } \Omega, \\
		M_i\Big(\mathcal{A}_i(u_i)\Big)\Big(|\nabla u_i|^{p(x)-2}\nabla u_i + \dfrac{|\nabla u_i|^{p(x)-2}\nabla u_i}{(1+ |\nabla u_i|^{p(x)})^{\frac{p(x)-2}{p(x)}}}\Big).\mathfrak{N}_i=|u_i|^{\ell_i (x)-2}u_i  \qquad\text{ on } \partial \Omega,
	\end{gathered}
\end{equation}
for $1\leq i\leq n (n\in \mathbb{N}^\ast)$, where
\[
\mathcal{A}_i(u_i)=\displaystyle\int_{\mathbb{R}^N }\frac{1}{p_i(x)}\Big(|\nabla u_i|^{p_i(x)} + \sqrt{1+|\nabla u_i|^{2p_i(x)}}+ |u_i|^{p_i(x)}\Big)dx.
\]
\end{example}
\begin{example}\label{example5.5}  
	Taking  $\displaystyle a_{1_i}(\xi)= 1+ \xi^{\frac{q_i(x)-p_i(x)}{p_i(x)}} +\frac{1}{(1+\xi)^{\frac{p_i(x)-2}{p_i(x)}}}$ and $\displaystyle a_{2_i}(\xi)= 1+ \xi^{\frac{q_i(x)-p_i(x)}{p_i(x)}}$, we see that $a_{1_i}$ satisfies the assumptions $(\textbf{\textit{A}}_1),(\textbf{\textit{A}}_2),$  $(\textbf{\textit{A}}_3),$ 
	and
	 $(\textbf{\textit{H}}_3),$ 
	 with $k^0_{1_i}=k^0_{2_i}=k^1_{2_i}=1$, 
	 $k^1_{1_i}= 2,$ 
	and 
	$k^3_{i}=k^2_{1_i}=k^2_{2_i}=1,$ for all $i\in \{1,2,...,n\}$.  
	Hence, system \eqref{s1.1}  becomes
\begin{equation}
	\label{exa5.5}
	\begin{gathered}
		-M_i\Big(\mathcal{A}_i(u_i)\Big)\Bigg(\Delta_{p_i(x)}u_i +\Delta_{q_i(x)}u_i +\textrm{div}\,\Big(\dfrac{|\nabla u_i|^{p_i(x)-2}\nabla u_i}{(1+ |\nabla u_i|^{p_i(x)})^{\frac{p_i(x)-2}{p_i(x)}}}\Big)- \Big(| u_i|^{p_i(x)-2} u_i +| u_i|^{q_i(x)-2} u_i \Big) \Bigg)\\
		= |u_i|^{s_i (x)-2}u_i +\lambda F_{u_i}(x,u)    \text{ in } \Omega, \\
		M_i\Big(\mathcal{A}_i(u_i)\Big)\Big(|\nabla u_i|^{p_i(x)-2}\nabla u_i+ \dfrac{|\nabla u_i|^{p(x)-2}\nabla u_i}{(1+ |\nabla u_i|^{p(x)})^{\frac{p(x)-2}{p(x)}}}  +
		|\nabla u_i|^{q_i(x)-2}\nabla u_i\Big).\mathfrak{N}_i=|u_i|^{\ell_i (x)-2}u_i  \qquad\text{ on } \partial \Omega,
	\end{gathered}
\end{equation}
for $1\leq i\leq n (n\in \mathbb{N}^\ast)$, where
\[
\mathcal{A}_i(u_i)=\displaystyle\int_{\Omega }\Bigg(\dfrac{1}{p_i(x)}\left(|\nabla u_i|^{p_i(x)}+ w_i(x)|u_i|^{p_i(x)} \right)+\dfrac{1}{q_i(x)}\left(|\nabla u_i|^{q_i(x)}+ w_i(x)|u_i|^{q_i(x)} \right) + \frac{1}{2}(1+ |\nabla u_i|^{p_i(x)})^{\frac{2}{p_i(x)}}\Bigg)dx.
\]
\end{example}
On the other hand, the class of systems \eqref{s1.1} can
contain one model of the above divergence operators, as in Examples \ref{example5.1}--\ref{example5.5}, or many different models of divergence operators simultaneously, depending on the phenomenon studied. Moreover, each equation in this class
can also be degenerate or nondegenerate. \\

\subsection*{Acknowledgements}  The second author acknowledges the funding received from the Slovenian Research Agency grants P1-0292, J1-4031, J1-4001, N1-0278, N1-0114, and N1-0083.
The authors  thank the referees for their suggestions and comments.

\end{document}